\documentclass[a4paper,12pt]{article}
\author{Johan Bredberg}
\title{Large gaps between consecutive zeros, on the critical line, of the Riemann zeta-function}
\usepackage{amsmath, amssymb, amsthm, cite}
\newtheorem{theorem}{Theorem}
\newtheorem{corollary}{Corollary}

\theoremstyle{definition}
\newtheorem{remark}{Remark}
\theoremstyle{definition}
\newtheorem{definition}{Definition}
\begin{document}
\date{ }
\maketitle
\textbf{Abstract.} 
We show that for any sufficiently large $T,$ there exists a subinterval of $[T,2T]$ of length at least $2.766 \times \frac{2\pi}{\log{T}},$ in which the function $t \mapsto \zeta(\frac{1}{2} + it)$ has no zeros.

\section{Introduction} \label{section-introduction} 

It is well-known (see for example \cite{tit} or \cite{dav}) that the Riemann zeta-function $\zeta(s)$ has so called trivial zeros at $s = -2, -4, -6,...$ and that all the other zeros $s = \sigma + it$ (the non-trivial ones) lie in the critical strip, i.e.\ satisfy $0 < \sigma < 1$. The Riemann Hypothesis (RH) is a conjecture saying that in fact all the non-trivial zeros must lie on the critical line, i.e.\ satisfy $\sigma = 1/2$. Levinson \cite{levins} showed that at least a third of the non-trivial zeros of the Riemann zeta-function lie on the critical line. However, to this day we do not know the whole truth about the horizontal distribution of the zeros. 

It is also well-known that the number of non-trivial zeros with ordinates in $[0,T]$ is $\frac{T \log{T}}{2 \pi} + O(T),$ 
which tells us that the average difference of the ordinates of two consecutive zeros at height $T$ is approximately $2\pi/\log{T}.$ Denote by $\{\gamma_{n}\}$ the sequence of ordinates of all zeros of $\zeta(s)$ in the upper halfplane, ordered in non-decreasing order. A natural question to ask is what one can say about 
\begin{equation*} 
\mu := \liminf_{n \to \infty} 
\frac{\gamma_{n+1}-\gamma_{n}}{(2\pi/\log \gamma_{n})} 
\quad \textrm{and  }  
\lambda := \limsup_{n \to \infty} 
\frac{\gamma_{n+1}-\gamma_{n}}{(2\pi/\log \gamma_{n})}.
\end{equation*}
In 1946, Selberg \cite{sel} remarked that $\mu < 1$ and $\lambda > 1.$ 
Although this is still the only known unconditional result, it is believed to be far from the whole truth. Indeed, in 1973, Montgomery \cite{Mont} predicted that $\mu = 0$ and $\lambda = \infty.$ On the assumption of RH, Feng and Wu \cite{FW} recently obtained $\mu \leqslant 0.5154$ and $\lambda \geqslant 2.7327.$ In this article only large gaps are considered and our overall strategy was first used by Hall \cite{hall-ett}. We show the following: 

\begin{theorem}[Main Theorem] \label{huvudsats} 
For any sufficiently large $T,$ there exists a subinterval of $[T,2T]$ of length at least $2.766 \times \frac{2\pi}{\log{T}},$ in which the function 
$t \mapsto \zeta(\frac{1}{2} + it)$ has no zeros.
\end{theorem}

\begin{remark} \label{average-gaps}
Notice that if we assume RH, then Theorem \ref{huvudsats} implies that 
$\lambda \geqslant 2.766.$
\end{remark}

\section{Building-stones in the proof of Theorem \ref{huvudsats}} 
\label{section-building-stones}

\subsection{Introducing the function $P(t,u,v,\kappa)$} 
\label{subsection-introducing-P}

\begin{definition} \label{f-def-hall}
Define 
\begin{equation} \label{f-def-hall-ekv}
P(t,u,v,\kappa) := \exp(v i \theta(t)) M(\tfrac{1}{2} + it) 
\zeta(\tfrac{1}{2} + it) 
\zeta(\tfrac{1}{2} + it + \tfrac{i \kappa}{\log{T}}), 
\end{equation} 
where 
\begin{equation} \label{theta-function-def}
\theta(t) := \textrm{Im}\Big(\log\Big(\Gamma\Big(\frac{1}{4} + \frac{it}{2}\Big)\Big)\Big) - \frac{t}{2} \log{\pi} 
\end{equation} 
and 
\begin{equation} \label{m-def-simplified}
M(s) := \sum_{h \leqslant T^{u}} \frac{1}{h^{s}},
\end{equation}
with $0 < u < 1/11.$ 
\end{definition}

\subsection{Main Assumption} \label{subsection-main-assumption}

We will now make an ``assumption". 
\\
\\
\textbf{Main Assumption}: Suppose that all the gaps between consecutive zeros of the function $t \mapsto P(t,u,v,\kappa)$ with 
$t \in [T,2T - \frac{\kappa}{\log{T}}]$ are\footnote{For either of the two zeros near the endpoints of the interval, we will here mean the distance from them to the respective endpoint.} at most $\frac{\kappa}{\log{T}}.$ 

\begin{remark} \label{Main-ass-remark}
For a suitable choice of $\kappa,$ $u$ and $v,$ we will eventually prove that our Main Assumption leads to a contradiction. 
\end{remark}

\subsection{Immediate consequences of our Main Assumption} 
\label{subsection-immediate-consequences}

Denote the zeros of $P(t,u,v,\kappa)$ with 
$T \leqslant t \leqslant 2T - \frac{\kappa}{\log{T}}$ by $t_{1}, t_{2},...,t_{N},$ ordered in non-decreasing order. Our Main Assumption implies that 
\begin{equation} \label{gap-gaps-ekv-sixth}
t_{i+1} - t_{i} 
\leqslant \frac{\kappa}{\log{T}}, 
\end{equation}
for $i = 1,2,...,N-1.$

\begin{remark} \label{andpunkter}
For future need we note here that our Main Assumption implies that 
\[ t_{1} \leqslant T + \frac{\kappa}{\log{T}} \textrm{ and } 
t_{N} \geqslant 2T - \frac{2\kappa}{\log{T}}. \]
\end{remark}

\subsection{Wirtinger's inequality and an application of it} 
\label{subsection-wirtinger}

We begin with the statement of the simplest version of Wirtinger's inequality. 

\begin{theorem} \label{wirtinger-ineq}
Suppose that $y(t)$ is a continuously differentiable function which satisfies 
$y(0) = y(\pi) = 0.$ Then 
\begin{equation} \label{wirtinger-complex-statement}
\int\limits_{0}^{\pi} |y(t)|^2 \; dt \leqslant 
\int\limits_{0}^{\pi} |y'(t)|^2 \; dt. 
\end{equation} 
\end{theorem}

\begin{proof}
For the case when $y(t)$ is a real-valued function, the reader is referred to Theorem 256 in Hardy, Littlewood and P\'{o}lya's \cite{HLP}. 

Say now that $y(t) = y_{1}(t) + iy_{2}(t),$ 
with $y_{1}$ and $y_{2}$ thus being real-valued continuously differentiable functions. Then clearly $y'(t) = y_{1}'(t) + iy_{2}'(t).$ 
What we want to show is 
\[ \int\limits_{0}^{\pi} y_{1}(t)^2 + y_{2}(t)^2 \; dt 
\leqslant \int\limits_{0}^{\pi} y_{1}'(t)^2 + y_{2}'(t)^2 \; dt, \]
but this immediately follows from the known (real) case.
\end{proof}

\begin{corollary} \label{wirtinger-ineq-scaled}
For $i = 1,2,...,N-1$ we have 
\begin{align*} 
\int\limits_{t_{i}}^{t_{i+1}} |P(t,u,v,\kappa)|^2 \; dt 
&\leqslant \Big( \frac{t_{i+1} - t_{i}}{\pi} \Big)^2
\int\limits_{t_{i}}^{t_{i+1}} |P'(t,u,v,\kappa)|^2 \; dt \\
&\leqslant \Big( \frac{\kappa}{\pi \log{T}} \Big)^2
\int\limits_{t_{i}}^{t_{i+1}} |P'(t,u,v,\kappa)|^2 \; dt. 
\end{align*}
\end{corollary}

\begin{proof}
One may make a linear substitution in Theorem \ref{wirtinger-ineq} to obtain a similar result if the function $y(t)$ has zeros at two general points $a$ and $b.$ We do so for the function $P(t,u,v,\kappa),$ which is continuously differentiable, and this gives us the first inequality. The latter inequality follows immediately from (\ref{gap-gaps-ekv-sixth}).
\end{proof}

Simply summing up the inequalities in Corollary \ref{wirtinger-ineq-scaled} for $i = 1,2,...,N-1,$ we obtain 
\begin{equation} \label{wirtinger-t-n-cor-wirtinger}
\int\limits_{t_{1}}^{t_{N}} |P(t,u,v,\kappa)|^2 \; dt 
\leqslant \Big( \frac{\kappa}{\pi \log{T}} \Big)^2 
\int\limits_{t_{1}}^{t_{N}} |P'(t,u,v,\kappa)|^2 \; dt. 
\end{equation} 

\subsection{Choosing weight-functions} \label{subsection-choosing-weight}

\begin{definition} \label{h-fn-def}
With $\eta > 0$ being any suitably small (fixed) constant, we define 
\begin{equation} \label{h-weighy-def}
h(x) := \left\{ \begin{array}{l l}
  \exp(-\eta T_{0} / x) & \quad \mbox{if $x > 0$,}\\
  0 & \quad \mbox{if $x \leqslant 0,$}\\
\end{array} \right. 
\end{equation}
with 
\begin{equation} \label{T-noll-val}
T_{0} = T^{1 - \epsilon}. 
\end{equation}
\end{definition}

Then clearly $h(x)$ is $C^{\infty}$ and $h(x) \leqslant 1.$ Also, 
\begin{equation*} \label{h-der-property}
h'(x) := \left\{ \begin{array}{l l}
  \eta T_{0}^{-1} (T_{0}/x)^{2} \exp(-\eta T_{0} / x) & \quad \mbox{if $x > 0$,}\\
  0 & \quad \mbox{if $x \leqslant 0,$}\\
\end{array} \right. 
\end{equation*}
which is seen to imply $h'(x) \ll T_{0}^{-1}.$ And more generally one finds that $h^{(j)}(x) \ll_{j} T_{0}^{-j}.$

Now take 
\begin{equation} \label{w-minus-def}
w_{-}(x) = h(x-T-T_{0}) h(2T-T_{0}-x)
\end{equation}
and
\begin{equation} \label{w-plus-def}
w_{+}(x) = \exp(2 \eta) h(x-T+T_{0}) h(2T+T_{0}-x).
\end{equation}
Remembering Remark \ref{andpunkter}, it is easily seen that (\ref{wirtinger-t-n-cor-wirtinger}) implies 

\begin{corollary} \label{wirtinger-sixth-val-cor}
\begin{equation} \label{wirtinger-sixth-w-val-ekv}
\int\limits_{-\infty}^{\infty} 
w_{-}(t) |P(t,u,v,\kappa)|^2 \; dt 
\leqslant \Big( \frac{\kappa}{\pi \log{T}} \Big)^2 
\int\limits_{-\infty}^{\infty} 
w_{+}(t) |P'(t,u,v,\kappa)|^2 \; dt. 
\end{equation}
\end{corollary}

\section{Going from Corollary \ref{wirtinger-sixth-val-cor} 
to Theorem \ref{huvudsats}} 
\label{section-going-from-to}

In this section we will write down asymptotic estimates for the Left Hand Side (LHS) and the Right Hand Side (RHS) in (\ref{wirtinger-sixth-w-val-ekv}), and use these to obtain an inequality in terms of $\kappa.$ 

\subsection{Giving names to some integrals} \label{subsection-giving-names}

Recall that $P(t,u,v,\kappa) = \exp(v i \theta(t)) M(\frac{1}{2} + it) 
\zeta(\frac{1}{2} + it) \zeta(\frac{1}{2} + it + \frac{i \kappa}{\log{T}}).$ Obviously 
\begin{equation} \label{f-squared}
|P(t,u,v,\kappa)|^{2} = |M(\tfrac{1}{2} + it)|^{2} 
|\zeta(\tfrac{1}{2} + it)|^{2} 
|\zeta(\tfrac{1}{2} + it + \tfrac{i \kappa}{\log{T}})|^{2}. 
\end{equation}
Next, using 
\begin{equation*} \label{m-der-into-n-etc}
M'(s) = - \sum_{h \leqslant T^{u}} \frac{\log{h}}{h^{s}} 
= N(s) - \log(T^{u}) M(s), 
\end{equation*}
where 
\begin{equation} \label{N-def-simplified}
N(s) := \sum_{h \leqslant T^{u}} \frac{\log(T^{u}/h)}{h^{s}}, 
\end{equation}
we find that 
\begin{align*} 
\frac{P'(t,u,v,\kappa)}{i \exp(v i \theta(t))} 
&= \big\{v \theta'(t) - \log(T^{u})\big\} M(\tfrac{1}{2} + it) 
\zeta(\tfrac{1}{2} + it) \zeta(\tfrac{1}{2} + it + \tfrac{i \kappa}{\log{T}}) \\
& + N(\tfrac{1}{2} + it) \zeta(\tfrac{1}{2} + it) \zeta(\tfrac{1}{2} + it + \tfrac{i \kappa}{\log{T}}) \nonumber \\
&+ M(\tfrac{1}{2} + it) 
\zeta'(\tfrac{1}{2} + it) \zeta(\tfrac{1}{2} + it + \tfrac{i \kappa}{\log{T}}) \nonumber \\
&+ M(\tfrac{1}{2} + it) 
\zeta(\tfrac{1}{2} + it) \zeta'(\tfrac{1}{2} + it + \tfrac{i \kappa}{\log{T}}). \nonumber
\end{align*}
Thus
\begin{align} \label{f-der-squared}
&|P'(t,u,v,\kappa)|^{2} 
= \big\{v \theta'(t) - \log(T^{u})\big\}^{2} |M(\tfrac{1}{2} + it)|^{2} 
|\zeta(\tfrac{1}{2} + it)|^{2} 
|\zeta(\tfrac{1}{2} + it + \tfrac{i \kappa}{\log{T}})|^{2} \\
& + 2 \big\{v \theta'(t) - \log(T^{u})\big\} \mathrm{Re}
\Big[ M(\tfrac{1}{2} + it) N(\tfrac{1}{2} - it)
|\zeta(\tfrac{1}{2} + it)|^{2} 
|\zeta(\tfrac{1}{2} + it + \tfrac{i \kappa}{\log{T}})|^{2} \Big] \nonumber \\ 
& + |N(\tfrac{1}{2} + it)|^{2} 
|\zeta(\tfrac{1}{2} + it)|^{2} 
|\zeta(\tfrac{1}{2} + it + \tfrac{i \kappa}{\log{T}})|^{2} \nonumber \\
& + 2 \big\{v \theta'(t) - \log(T^{u})\big\} \mathrm{Re}
\Big[ |M(\tfrac{1}{2} + it)|^{2}
\zeta'(\tfrac{1}{2} + it) \zeta(\tfrac{1}{2} - it)
|\zeta(\tfrac{1}{2} + it + \tfrac{i \kappa}{\log{T}})|^{2} \Big] \nonumber \\
& + 2 \big\{v \theta'(t) - \log(T^{u})\big\} \mathrm{Re}
\Big[ |M(\tfrac{1}{2} + it)|^{2}
|\zeta(\tfrac{1}{2} + it)|^{2} 
\zeta'(\tfrac{1}{2} + it + \tfrac{i \kappa}{\log{T}}) 
\zeta(\tfrac{1}{2} - it - \tfrac{i \kappa}{\log{T}}) \Big] \nonumber \\
& + |M(\tfrac{1}{2} + it)|^{2} 
|\zeta'(\tfrac{1}{2} + it)|^{2} 
|\zeta(\tfrac{1}{2} + it + \tfrac{i \kappa}{\log{T}})|^{2} \nonumber \\
& + |M(\tfrac{1}{2} + it)|^{2} 
|\zeta(\tfrac{1}{2} + it)|^{2} 
|\zeta'(\tfrac{1}{2} + it + \tfrac{i \kappa}{\log{T}})|^{2} \nonumber \\
& + 2 \mathrm{Re} \Big[ |M(\tfrac{1}{2} + it)|^{2} 
\zeta'(\tfrac{1}{2} + it) \zeta(\tfrac{1}{2} - it) 
\zeta(\tfrac{1}{2} + it + \tfrac{i \kappa}{\log{T}}) 
\zeta'(\tfrac{1}{2} - it - \tfrac{i \kappa}{\log{T}}) \Big] \nonumber \\
& + 2 \mathrm{Re} \Big[ M(\tfrac{1}{2} + it) N(\tfrac{1}{2} - it) 
\zeta'(\tfrac{1}{2} + it) \zeta(\tfrac{1}{2} - it) 
|\zeta(\tfrac{1}{2} + it + \tfrac{i \kappa}{\log{T}})|^{2} \Big] \nonumber \\
& + 2 \mathrm{Re} \Big[ M(\tfrac{1}{2} + it) N(\tfrac{1}{2} - it) 
|\zeta(\tfrac{1}{2} + it)|^{2} 
\zeta'(\tfrac{1}{2} + it + \tfrac{i \kappa}{\log{T}}) 
\zeta(\tfrac{1}{2} - it - \tfrac{i \kappa}{\log{T}}) \Big]. \nonumber 
\end{align}

For future need we now introduce some notation. Let 
\begin{equation} \label{integral-namn-A}
A_{\pm} = \int\limits_{-\infty}^{\infty} 
w_{\pm}(t) |M(\tfrac{1}{2} + it)|^{2} 
|\zeta(\tfrac{1}{2} + it)|^{2} 
|\zeta(\tfrac{1}{2} + it + \tfrac{i \kappa}{\log{T}})|^{2} \; dt, 
\end{equation}
\begin{equation} \label{integral-namn-B}
B_{\pm} = \int\limits_{-\infty}^{\infty} 
w_{\pm}(t) \mathrm{Re}
\Big[ M(\tfrac{1}{2} + it) N(\tfrac{1}{2} - it)
|\zeta(\tfrac{1}{2} + it)|^{2} 
|\zeta(\tfrac{1}{2} + it + \tfrac{i \kappa}{\log{T}})|^{2} \Big] \; dt, 
\end{equation}
\begin{equation} \label{integral-namn-C}
C_{\pm} = \int\limits_{-\infty}^{\infty} 
w_{\pm}(t) |N(\tfrac{1}{2} + it)|^{2} 
|\zeta(\tfrac{1}{2} + it)|^{2} 
|\zeta(\tfrac{1}{2} + it + \tfrac{i \kappa}{\log{T}})|^{2} \; dt, 
\end{equation}
\begin{equation} \label{integral-namn-D}
D_{\pm} = \int\limits_{-\infty}^{\infty} 
w_{\pm}(t) \mathrm{Re}
\Big[ |M(\tfrac{1}{2} + it)|^{2}
\zeta'(\tfrac{1}{2} + it) \zeta(\tfrac{1}{2} - it)
|\zeta(\tfrac{1}{2} + it + \tfrac{i \kappa}{\log{T}})|^{2} \Big] \; dt, 
\end{equation}
\begin{equation} \label{integral-namn-E}
E_{\pm} = \int\limits_{-\infty}^{\infty} 
w_{\pm}(t) \mathrm{Re}
\Big[ |M(\tfrac{1}{2} + it)|^{2}
|\zeta(\tfrac{1}{2} + it)|^{2} 
\zeta'(\tfrac{1}{2} + it + \tfrac{i \kappa}{\log{T}}) 
\zeta(\tfrac{1}{2} - it - \tfrac{i \kappa}{\log{T}}) \Big] \; dt, 
\end{equation}
\begin{equation} \label{integral-namn-F}
F_{\pm} = \int\limits_{-\infty}^{\infty} 
w_{\pm}(t) 
|M(\tfrac{1}{2} + it)|^{2} 
|\zeta'(\tfrac{1}{2} + it)|^{2} 
|\zeta(\tfrac{1}{2} + it + \tfrac{i \kappa}{\log{T}})|^{2} \; dt, 
\end{equation}
\begin{equation} \label{integral-namn-G}
G_{\pm} = \int\limits_{-\infty}^{\infty} 
w_{\pm}(t) |M(\tfrac{1}{2} + it)|^{2} 
|\zeta(\tfrac{1}{2} + it)|^{2} 
|\zeta'(\tfrac{1}{2} + it + \tfrac{i \kappa}{\log{T}})|^{2} \; dt, 
\end{equation}
\begin{equation} \label{integral-namn-H}
H_{\pm} = \int\limits_{-\infty}^{\infty} 
w_{\pm}(t) \mathrm{Re} \Big[ |M(\tfrac{1}{2} + it)|^{2} 
\zeta'(\tfrac{1}{2} + it) \zeta(\tfrac{1}{2} - it) 
\zeta(\tfrac{1}{2} + it + \tfrac{i \kappa}{\log{T}}) 
\zeta'(\tfrac{1}{2} - it - \tfrac{i \kappa}{\log{T}}) \Big] \; dt, 
\end{equation}
\begin{equation} \label{integral-namn-I}
I_{\pm} = \int\limits_{-\infty}^{\infty} 
w_{\pm}(t) \mathrm{Re} \Big[ M(\tfrac{1}{2} + it) N(\tfrac{1}{2} - it) 
\zeta'(\tfrac{1}{2} + it) \zeta(\tfrac{1}{2} - it) 
|\zeta(\tfrac{1}{2} + it + \tfrac{i \kappa}{\log{T}})|^{2} \Big] \; dt 
\end{equation}
and
\begin{equation} \label{integral-namn-J}
J_{\pm} = \int\limits_{-\infty}^{\infty} 
w_{\pm}(t) \mathrm{Re} \Big[ M(\tfrac{1}{2} + it) N(\tfrac{1}{2} - it) 
|\zeta(\tfrac{1}{2} + it)|^{2} 
\zeta'(\tfrac{1}{2} + it + \tfrac{i \kappa}{\log{T}}) 
\zeta(\tfrac{1}{2} - it - \tfrac{i \kappa}{\log{T}}) \Big] \; dt. 
\end{equation}

\subsection{Evaluation of the integrals defined in Section \ref{subsection-giving-names}} 
\label{subsection-evaluation-svar}

The reader is referred to Section \ref{section-evaluation-integrals} for details on how to evaluate (asymptotically) the integrals (\ref{integral-namn-A})-(\ref{integral-namn-J}). Below we give the answers. 

\begin{equation} \label{integral-namn-A-svar}
A_{\pm} = A_{\kappa} \cdot a_{3} \cdot 
\Big[ \int_{-\infty}^{\infty} w_{\pm}(t) \; dt \Big] 
\cdot \log^{9}{T} + O(T \log^{8}{T}), 
\end{equation}
\begin{equation} \label{integral-namn-B-svar}
B_{\pm} = B_{\kappa} \cdot a_{3} \cdot 
\Big[ \int_{-\infty}^{\infty} w_{\pm}(t) \; dt \Big] 
\cdot \log^{10}{T} + O(T \log^{9}{T}), 
\end{equation}
\begin{equation} \label{integral-namn-C-svar}
C_{\pm} = C_{\kappa} \cdot a_{3} \cdot 
\Big[ \int_{-\infty}^{\infty} w_{\pm}(t) \; dt \Big] 
\cdot \log^{11}{T} + O(T \log^{10}{T}), 
\end{equation}
\begin{equation} \label{integral-namn-D-svar}
D_{\pm} = D_{\kappa} \cdot a_{3} \cdot 
\Big[ \int_{-\infty}^{\infty} w_{\pm}(t) \; dt \Big] 
\cdot \log^{10}{T} + O(T \log^{9}{T}), 
\end{equation}
\begin{equation} \label{integral-namn-E-svar}
E_{\pm} = E_{\kappa} \cdot a_{3} \cdot 
\Big[ \int_{-\infty}^{\infty} w_{\pm}(t) \; dt \Big] 
\cdot \log^{10}{T} + O(T \log^{9}{T}), 
\end{equation}
\begin{equation} \label{integral-namn-F-svar}
F_{\pm} = F_{\kappa} \cdot a_{3} \cdot 
\Big[ \int_{-\infty}^{\infty} w_{\pm}(t) \; dt \Big] 
\cdot \log^{11}{T} + O(T \log^{10}{T}), 
\end{equation}
\begin{equation} \label{integral-namn-G-svar}
G_{\pm} = G_{\kappa} \cdot a_{3} \cdot 
\Big[ \int_{-\infty}^{\infty} w_{\pm}(t) \; dt \Big] 
\cdot \log^{11}{T} + O(T \log^{10}{T}), 
\end{equation}
\begin{equation} \label{integral-namn-H-svar}
H_{\pm} = H_{\kappa} \cdot a_{3} \cdot 
\Big[ \int_{-\infty}^{\infty} w_{\pm}(t) \; dt \Big] 
\cdot \log^{11}{T} + O(T \log^{10}{T}), 
\end{equation}
\begin{equation} \label{integral-namn-I-svar}
I_{\pm} = I_{\kappa} \cdot a_{3} \cdot 
\Big[ \int_{-\infty}^{\infty} w_{\pm}(t) \; dt \Big] 
\cdot \log^{11}{T} + O(T \log^{10}{T}) 
\end{equation}
and
\begin{equation} \label{integral-namn-J-svar}
J_{\pm} = J_{\kappa} \cdot a_{3} \cdot 
\Big[ \int_{-\infty}^{\infty} w_{\pm}(t) \; dt \Big] 
\cdot \log^{11}{T} + O(T \log^{10}{T}), 
\end{equation}
with 
\begin{equation} \label{L-definitionen}
a_{3} = \displaystyle\prod_{p} 
\Big\{ \Big(1 + \frac{4}{p} + \frac{1}{p^{2}}\Big)
\Big(1 - \frac{1}{p}\Big)^{4} \Big\}
\end{equation}
and where the constants are given by
\begin{align} \label{integral-namn-A-svarrrr}
&A_{\kappa} = \frac{(-10)}{\kappa^{8}} 
+ \frac{(2u - 2u^{2} + \frac{u^{3}}{3})}{\kappa^{6}} 
+ \frac{(\frac{u^{3}}{3} - \frac{u^{4}}{4})}{\kappa^{4}} 
+ \frac{8 \sin(\kappa u)}{\kappa^{9}} 
+ \frac{(10 - 8u) \cos(\kappa u)}{\kappa^{8}} \\
&+ \frac{(-4 + 10u - 4u^{2}) \sin(\kappa u)}{\kappa^{7}} 
+ \frac{(2u - 3u^{2} + u^{3}) \cos(\kappa u)}{\kappa^{6}}  
+ \frac{(-8) \sin{\kappa}}{\kappa^{9}}  
+ \frac{(\frac{-u^{3}}{3}) \cos{\kappa}}{\kappa^{6}} \nonumber \\ 
&+ \frac{8 \sin(\kappa (1 - u))}{\kappa^{9}} 
+ \frac{(8u) \cos(\kappa (1 - u))}{\kappa^{8}} 
+ \frac{(-4u^{2}) \sin(\kappa (1 - u))}{\kappa^{7}}  
+ \frac{(-u^{3}) \cos(\kappa (1 - u))}{\kappa^{6}}, \nonumber  
\end{align}
\begin{align} \label{integral-namn-B-svarrrr}
&B_{\kappa} = \frac{(1 - 2u)}{\kappa^{8}} 
+ \frac{(- \frac{u^{3}}{3} + \frac{u^{4}}{6})}{\kappa^{6}} 
+ \frac{(\frac{u^{4}}{8} - \frac{u^{5}}{12})}{\kappa^{4}} 
+ \frac{(-1 + 2u) \cos(\kappa u)}{\kappa^{8}} \\
&+ \frac{(-u + 2u^{2} - \frac{u^{3}}{3}) \sin(\kappa u)}{\kappa^{7}} 
+ \frac{(\frac{u^{2}}{2} - \frac{2u^{3}}{3} + \frac{u^{4}}{6}) 
\cos(\kappa u)}{\kappa^{6}} 
+ \frac{(\frac{u^{3}}{3}) \sin{\kappa}}{\kappa^{7}}  
+ \frac{(\frac{-u^{4}}{6}) \cos{\kappa}}{\kappa^{6}} \nonumber \\ 
&+ \frac{(-\frac{u^{3}}{3}) \sin(\kappa (1 - u))}{\kappa^{7}}  
+ \frac{(-\frac{u^{4}}{6}) \cos(\kappa (1 - u))}{\kappa^{6}}, \nonumber  
\end{align}
\begin{align} \label{integral-namn-C-svarrrr}
&C_{\kappa} = \frac{(-20)}{\kappa^{10}} 
+ \frac{(2u - 2u^{2})}{\kappa^{8}} 
+ \frac{(-\frac{u^{4}}{6} + \frac{u^{5}}{15})}{\kappa^{6}} 
+ \frac{(\frac{u^{5}}{20} - \frac{u^{6}}{36})}{\kappa^{4}} 
+ \frac{12 \sin(\kappa u)}{\kappa^{11}} \\ 
&+ \frac{(20 - 12u) \cos(\kappa u)}{\kappa^{10}} 
+ \frac{(-6 + 20u - 6u^{2}) \sin(\kappa u)}{\kappa^{9}} 
+ \frac{(4u - 8u^{2} + 2u^{3}) \cos(\kappa u)}{\kappa^{8}} \nonumber \\ 
&+ \frac{(u^{2} - \frac{4u^{3}}{3} + \frac{u^{4}}{3}) 
\sin(\kappa u)}{\kappa^{7}} 
+ \frac{(-12) \sin{\kappa}}{\kappa^{11}}  
+ \frac{(\frac{u^{4}}{6}) \sin{\kappa}}{\kappa^{7}} 
+ \frac{(-\frac{u^{5}}{15}) \cos{\kappa}}{\kappa^{6}} \nonumber \\ 
&+ \frac{12 \sin(\kappa (1 - u))}{\kappa^{11}} 
+ \frac{(12u) \cos(\kappa (1 - u))}{\kappa^{10}} 
+ \frac{(-6u^{2}) \sin(\kappa (1 - u))}{\kappa^{9}} \nonumber \\ 
&+ \frac{(-2u^{3}) \cos(\kappa (1 - u))}{\kappa^{8}} 
+ \frac{(\frac{u^{4}}{3}) \sin(\kappa (1 - u))}{\kappa^{7}}, \nonumber  
\end{align}
\begin{equation} \label{integral-namn-D-svarrrr}
D_{\kappa} = -\frac{A_{\kappa}}{2}, \qquad \qquad \qquad \qquad \qquad \qquad \qquad \qquad \qquad \qquad \qquad \qquad \qquad  
\end{equation}
\begin{equation} \label{integral-namn-E-svarrrr}
E_{\kappa} = -\frac{A_{\kappa}}{2}, \qquad \qquad \qquad \qquad \qquad \qquad \qquad \qquad \qquad \qquad \qquad \qquad \qquad  
\end{equation}
\begin{align} \label{integral-namn-F-svarrrr}
&F_{\kappa} = \frac{(-66)}{\kappa^{10}} 
+ \frac{(-\frac{11}{3} + 8u - 8u^{2} + \frac{4u^{3}}{3})}{\kappa^{8}} 
+ \frac{(\frac{2u}{3} - \frac{2u^{2}}{3} + \frac{5u^{3}}{6} 
- \frac{2u^{4}}{3} - \frac{u^{5}}{60})}{\kappa^{6}} \\ 
&+ \frac{(\frac{u^{3}}{9} - \frac{u^{4}}{8} + \frac{u^{5}}{20} 
- \frac{u^{6}}{72})}{\kappa^{4}} 
+ \frac{84 \sin(\kappa u)}{\kappa^{11}} 
+ \frac{(66 - 84u) \cos(\kappa u)}{\kappa^{10}} 
+ \frac{(-16 + 66u - 42u^{2}) \sin(\kappa u)}{\kappa^{9}} \nonumber \\ 
&+ \frac{(\frac{11}{3} + 8u - 25u^{2} + \frac{38u^{3}}{3}) 
\cos(\kappa u)}{\kappa^{8}} 
+ \frac{(-\frac{4}{3} + \frac{11u}{3} - \frac{11u^{3}}{3} 
+ \frac{11u^{4}}{6}) \sin(\kappa u)}{\kappa^{7}} \nonumber \\ 
&+ \frac{(\frac{2u}{3} - \frac{7u^{2}}{6} + \frac{u^{3}}{2} 
+ \frac{u^{4}}{12} - \frac{u^{5}}{12}) \cos(\kappa u)}{\kappa^{6}} 
+ \frac{(-84) \sin{\kappa}}{\kappa^{11}} 
+ \frac{26 \cos{\kappa}}{\kappa^{10}} 
+ \frac{(-\frac{4u^{3}}{3}) \cos{\kappa}}{\kappa^{8}} \nonumber \\ 
&+ \frac{(-\frac{2u^{3}}{3} + \frac{u^{4}}{3}) \sin{\kappa}}{\kappa^{7}} 
+ \frac{(-\frac{u^{4}}{12} + \frac{u^{5}}{60}) \cos{\kappa}}{\kappa^{6}}  
+ \frac{84 \sin(\kappa (1 - u))}{\kappa^{11}} 
+ \frac{(-26 + 84u) \cos(\kappa (1 - u))}{\kappa^{10}} \nonumber \\ 
&+ \frac{(26u - 42u^{2}) \sin(\kappa (1 - u))}{\kappa^{9}} 
+ \frac{(13u^{2} - \frac{38u^{3}}{3}) \cos(\kappa (1 - u))}{\kappa^{8}} \nonumber \\
&+ \frac{(-\frac{11u^{3}}{3} + \frac{11u^{4}}{6}) 
\sin(\kappa (1 - u))}{\kappa^{7}} 
+ \frac{(-\frac{u^{4}}{3} + \frac{u^{5}}{12}) \cos(\kappa (1 - u))}{\kappa^{6}}, \nonumber  
\end{align}
\begin{align} \label{integral-namn-G-svarrrr}
&G_{\kappa} = \frac{(-148)}{\kappa^{10}} 
+ \frac{(-\frac{14}{3} + 18u - 18u^{2} + 3u^{3})}{\kappa^{8}} 
+ \frac{(\frac{2u}{3} - u^{2} + \frac{11u^{3}}{6} 
- \frac{7u^{4}}{6})}{\kappa^{6}} 
+ \frac{(\frac{u^{3}}{9} - \frac{u^{4}}{12})}{\kappa^{4}} \\ 
&+ \frac{152 \sin(\kappa u)}{\kappa^{11}} 
+ \frac{(148 - 152u) \cos(\kappa u)}{\kappa^{10}} 
+ \frac{(-40 + 148u - 76u^{2}) \sin(\kappa u)}{\kappa^{9}} \nonumber \\ 
&+ \frac{(\frac{14}{3} + 22u - 56u^{2} + \frac{67u^{3}}{3}) 
\cos(\kappa u)}{\kappa^{8}} 
+ \frac{(-\frac{4}{3} + \frac{14u}{3} + 2u^{2} - \frac{26u^{3}}{3} 
+ \frac{10u^{4}}{3}) \sin(\kappa u)}{\kappa^{7}} \nonumber \\ 
&+ \frac{(\frac{2u}{3} - \frac{4u^{2}}{3} + \frac{u^{3}}{2} 
+ \frac{u^{4}}{3} - \frac{u^{5}}{6}) \cos(\kappa u)}{\kappa^{6}} 
+ \frac{(-152) \sin{\kappa}}{\kappa^{11}} 
+ \frac{36 \cos{\kappa}}{\kappa^{10}} 
+ \frac{(-3u^{3}) \cos{\kappa}}{\kappa^{8}} \nonumber \\ 
&+ \frac{(-u^{3}) \sin{\kappa}}{\kappa^{7}} 
+ \frac{152 \sin(\kappa (1 - u))}{\kappa^{11}} 
+ \frac{(-36 + 152u) \cos(\kappa (1 - u))}{\kappa^{10}} \nonumber \\ 
&+ \frac{(36u - 76u^{2}) \sin(\kappa (1 - u))}{\kappa^{9}} 
+ \frac{(18u^{2} - \frac{67u^{3}}{3}) \cos(\kappa (1 - u))}{\kappa^{8}} \nonumber \\ 
&+ \frac{(-5u^{3} + \frac{10u^{4}}{3}) \sin(\kappa (1 - u))}{\kappa^{7}} 
+ \frac{(-\frac{u^{4}}{2} + \frac{u^{5}}{6}) \cos(\kappa (1 - u))}{\kappa^{6}}, \nonumber  
\end{align}
\begin{align} \label{integral-namn-H-svarrrr}
&H_{\kappa} = \frac{117}{\kappa^{10}} 
+ \frac{(-\frac{5}{2} - 14u + 14u^{2} - \frac{7u^{3}}{3})}{\kappa^{8}} 
+ \frac{(\frac{u}{2} - \frac{u^{2}}{2} - \frac{7u^{3}}{6} 
+ \frac{25u^{4}}{24})}{\kappa^{6}} 
+ \frac{(\frac{u^{3}}{12} - \frac{u^{4}}{16})}{\kappa^{4}} \\ 
&+ \frac{(-130) \sin(\kappa u)}{\kappa^{11}} 
+ \frac{(-117 + 130u) \cos(\kappa u)}{\kappa^{10}} 
+ \frac{(38 - 117u + 65u^{2}) \sin(\kappa u)}{\kappa^{9}} \nonumber \\ 
&+ \frac{(\frac{5}{2} - 24u + \frac{89u^{2}}{2} - \frac{58u^{3}}{3}) 
\cos(\kappa u)}{\kappa^{8}} 
+ \frac{(-1 + \frac{5u}{2} - 5u^{2} + 7u^{3} 
- \frac{17u^{4}}{6}) \sin(\kappa u)}{\kappa^{7}} \nonumber \\ 
&+ \frac{(\frac{u}{2} - \frac{3u^{2}}{4} + \frac{u^{3}}{2} 
- \frac{5u^{4}}{12} + \frac{u^{5}}{6}) \cos(\kappa u)}{\kappa^{6}} 
+ \frac{130 \sin{\kappa}}{\kappa^{11}} 
+ \frac{(-31) \cos{\kappa}}{\kappa^{10}} 
+ \frac{(-4) \sin{\kappa}}{\kappa^{9}} \nonumber \\ 
&+ \frac{(\frac{7u^{3}}{3}) \cos{\kappa}}{\kappa^{8}} 
+ \frac{(\frac{5u^{3}}{6} - \frac{u^{4}}{4}) \sin{\kappa}}{\kappa^{7}} 
+ \frac{(-\frac{u^{3}}{6} + \frac{u^{4}}{24}) \cos{\kappa}}{\kappa^{6}} 
+ \frac{(-130) \sin(\kappa (1 - u))}{\kappa^{11}} \nonumber \\ 
&+ \frac{(31 - 130u) \cos(\kappa (1 - u))}{\kappa^{10}} 
+ \frac{(4 - 31u + 65u^{2}) \sin(\kappa (1 - u))}{\kappa^{9}} \nonumber \\
&+ \frac{(4u - \frac{31u^{2}}{2} + \frac{58u^{3}}{3}) 
\cos(\kappa (1 - u))}{\kappa^{8}}  
+ \frac{(-2u^{2} + \frac{13u^{3}}{3} - \frac{17u^{4}}{6}) 
\sin(\kappa (1 - u))}{\kappa^{7}} \nonumber \\ 
&+ \frac{(-\frac{u^{3}}{2} + \frac{5u^{4}}{12} - \frac{u^{5}}{6}) 
\cos(\kappa (1 - u))}{\kappa^{6}}, \nonumber  
\end{align}
\begin{align} \label{integral-namn-I-svarrrr}
&I_{\kappa} = \frac{(-35)}{\kappa^{10}} 
+ \frac{(-\frac{1}{2} + 5u - 4u^{2} + \frac{2u^{3}}{3})}{\kappa^{8}} 
+ \frac{(\frac{5u^{3}}{12} - \frac{u^{4}}{4} + \frac{u^{5}}{40})}{\kappa^{6}} 
+ \frac{(-\frac{u^{4}}{16} + \frac{u^{5}}{20} 
- \frac{u^{6}}{144})}{\kappa^{4}} \\ 
&+ \frac{32 \sin(\kappa u)}{\kappa^{11}} 
+ \frac{(35 - 32u) \cos(\kappa u)}{\kappa^{10}} 
+ \frac{(-11 + 35u - 16u^{2}) \sin(\kappa u)}{\kappa^{9}} \nonumber \\ 
&+ \frac{(\frac{1}{2} + 6u - \frac{27u^{2}}{2} + \frac{14u^{3}}{3}) 
\cos(\kappa u)}{\kappa^{8}} 
+ \frac{(\frac{u}{2} + \frac{u^{2}}{2} - \frac{13u^{3}}{6} 
+ \frac{3u^{4}}{4}) \sin(\kappa u)}{\kappa^{7}} \nonumber \\ 
&+ \frac{(-\frac{u^{2}}{4} + \frac{u^{3}}{4} 
+ \frac{u^{4}}{24} - \frac{u^{5}}{24}) \cos(\kappa u)}{\kappa^{6}} 
+ \frac{(-32) \sin{\kappa}}{\kappa^{11}} 
+ \frac{5 \cos{\kappa}}{\kappa^{10}} 
+ \frac{(-\frac{2u^{3}}{3}) \cos{\kappa}}{\kappa^{8}} \nonumber \\ 
&+ \frac{(-\frac{u^{3}}{3} - \frac{u^{4}}{12}) \sin{\kappa}}{\kappa^{7}} 
+ \frac{(\frac{u^{4}}{8} - \frac{u^{5}}{40}) \cos{\kappa}}{\kappa^{6}} 
+ \frac{32 \sin(\kappa (1 - u))}{\kappa^{11}}  
+ \frac{(-5 + 32u) \cos(\kappa (1 - u))}{\kappa^{10}} \nonumber \\
&+ \frac{(5u - 16u^{2}) \sin(\kappa (1 - u))}{\kappa^{9}} 
+ \frac{(\frac{5u^{2}}{2} - \frac{14u^{3}}{3}) 
\cos(\kappa (1 - u))}{\kappa^{8}} 
+ \frac{(-\frac{u^{3}}{2} + \frac{3u^{4}}{4}) 
\sin(\kappa (1 - u))}{\kappa^{7}} \nonumber \\ 
&+ \frac{(\frac{u^{5}}{24}) \cos(\kappa (1 - u))}{\kappa^{6}} \nonumber  
\end{align}
and finally 
\begin{align} \label{integral-namn-J-svarrrr}
&J_{\kappa} = \frac{63}{\kappa^{10}} 
+ \frac{(-\frac{1}{2} - 6u + 7u^{2} - \frac{7u^{3}}{6})}{\kappa^{8}} 
+ \frac{(-\frac{u^{3}}{4} + \frac{u^{4}}{8})}{\kappa^{6}} 
+ \frac{(-\frac{u^{4}}{16} + \frac{u^{5}}{24})}{\kappa^{4}} 
+ \frac{(-50) \sin(\kappa u)}{\kappa^{11}} \\
&+ \frac{(-63 + 50u) \cos(\kappa u)}{\kappa^{10}} 
+ \frac{(20 - 63u + 25u^{2}) \sin(\kappa u)}{\kappa^{9}} 
+ \frac{(\frac{1}{2} - 14u + \frac{49u^{2}}{2} - \frac{43u^{3}}{6}) 
\cos(\kappa u)}{\kappa^{8}} \nonumber \\
&+ \frac{(\frac{u}{2} - 4u^{2} + \frac{14u^{3}}{3} 
- \frac{7u^{4}}{6}) \sin(\kappa u)}{\kappa^{7}} 
+ \frac{(-\frac{u^{2}}{4} + \frac{7u^{3}}{12} 
- \frac{5u^{4}}{12} + \frac{u^{5}}{12}) \cos(\kappa u)}{\kappa^{6}} 
+ \frac{50 \sin{\kappa}}{\kappa^{11}} \nonumber \\ 
&+ \frac{(-5) \cos{\kappa}}{\kappa^{10}} 
+ \frac{(\frac{7u^{3}}{6}) \cos{\kappa}}{\kappa^{8}} 
+ \frac{(\frac{u^{4}}{4}) \sin{\kappa}}{\kappa^{7}} 
+ \frac{(\frac{u^{4}}{24}) \cos{\kappa}}{\kappa^{6}} 
+ \frac{(-50) \sin(\kappa (1 - u))}{\kappa^{11}} \nonumber \\ 
&+ \frac{(5 - 50u) \cos(\kappa (1 - u))}{\kappa^{10}} 
+ \frac{(-5u + 25u^{2}) \sin(\kappa (1 - u))}{\kappa^{9}} 
+ \frac{(-\frac{5u^{2}}{2} + \frac{43u^{3}}{6}) 
\cos(\kappa (1 - u))}{\kappa^{8}} \nonumber \\ 
&+ \frac{(\frac{5u^{3}}{6} - \frac{7u^{4}}{6}) 
\sin(\kappa (1 - u))}{\kappa^{7}} 
+ \frac{(\frac{u^{4}}{6} - \frac{u^{5}}{12}) \cos(\kappa (1 - u))}{\kappa^{6}}. \nonumber  
\end{align}

\begin{remark} \label{cancellation-forstas}
If the above ten coefficients are seen as Laurent series in terms of $\kappa,$ then numerical calculations show that all coefficients for negative $\kappa$-powers equal zero. This had to be the case since our expressions are analytic in $\kappa.$ The latter can be seen from the fact that e.g.\ the LHS of (\ref{integral-namn-A-svar}) remains bounded if we let 
$\kappa \to 0.$ 
\end{remark}

\begin{remark} \label{A-kappa-u-ett}
When\footnote{Although the results in this article (via \cite{HY}) only are shown for $u < 1/11,$ it may be that they hold for $u < 1.$} we put $u = 1,$ the limit of $A_{\kappa}$ as $\kappa \to 0$ equals $\frac{42}{9!}.$ If we let the weight-function $w(t)$ be an approximation to the characteristic function on $[T,2T],$ then (\ref{integral-namn-A-svar}) is seen to be consistent with the conjecture that 
\begin{equation} \label{sixth-conj-ekv}
\int\limits_{T}^{2T} |\zeta(\tfrac{1}{2} + it)|^{6} \; dt 
\sim \frac{42}{9!} \cdot a_{3} \cdot T \log^{9}{T} 
\end{equation}
which of course is good news\footnote{However, putting 
$u = 1/2$ and letting $\kappa \to 0$ does not give half of the sixth power moment.}.
\end{remark}

\subsection{Obtaining an inequality in terms of $\kappa$} 
\label{subsection-obtaining-ineq}

It is now time to investigate both sides of Corollary \ref{wirtinger-sixth-val-cor}. Focusing on the RHS, we are lead to recall (\ref{f-der-squared}). Since $w_{+}(t)$ is supported in $[\frac{T}{2},4T]$ we may use that 
\begin{equation} \label{theta-der-prop}
\theta'(t) = \frac{\log{T}}{2} + O(1). 
\end{equation} 
The contribution to the integral in the RHS of (\ref{wirtinger-sixth-w-val-ekv}) coming from the error term in (\ref{theta-der-prop}) can be seen to be 
$\ll T \log^{10}{T}.$ To do this we simply use the Cauchy--Schwarz inequality 
\begin{equation*} \label{CS-ineq}
\Big| \int f(x) g(x) \; dx \Big|^{2} 
\leqslant \int |f(x)|^{2} \; dx \int |g(x)|^{2} \; dx.
\end{equation*} 

We can thus via (\ref{f-squared}) and (\ref{f-der-squared}) convert both sides of (\ref{wirtinger-sixth-w-val-ekv}) into expressions involving the integrals in (\ref{integral-namn-A})-(\ref{integral-namn-J}). The latter are of course evaluated using (\ref{integral-namn-A-svar})-(\ref{integral-namn-J-svar}). Explicitly this procedure yields 
\begin{align} \label{olikhet-sixth-w-fortf}
&A_{\kappa} \cdot a_{3} \cdot 
\Big[ \int_{-\infty}^{\infty} w_{-}(t) \; dt \Big] 
\cdot \log^{9}{T} + O(T \log^{8}{T}) \\ 
& \leqslant \Big(\frac{\kappa}{\pi}\Big)^{2} 
\cdot \Big\{ (v/2 - u)^{2} A_{\kappa} 
+ (v - 2u) B_{\kappa} 
+ C_{\kappa} + 
(v - 2u) D_{\kappa} 
+ (v - 2u)E_{\kappa} \nonumber \\
&+ F_{\kappa} 
+ G_{\kappa} 
+ 2 H_{\kappa} 
+ 2 I_{\kappa} 
+ 2 J_{\kappa} \Big\} \cdot a_{3} \cdot 
\Big[ \int_{-\infty}^{\infty} w_{+}(t) \; dt \Big] 
\cdot \log^{9}{T} + O(T \log^{8}{T}). \nonumber
\end{align}

Immediately from the definitons (\ref{w-minus-def}) and (\ref{w-plus-def}) one sees 
\begin{equation} \label{w-minus-ineq}
w_{-}(t) \geqslant \exp(-2\eta) \chi_{[T + 2T_{0},2T - 2T_{0}]}(t)
\end{equation}
and
\begin{equation} \label{w-plus-ineq}
w_{+}(t) \leqslant \exp(2\eta) \chi_{[T - T_{0},2T + T_{0}]}(t). 
\end{equation}
Therefore for sufficiently large $T$ we have (recall (\ref{T-noll-val})) 
\begin{equation} \label{w-minus-integral-ineq}
\int_{-\infty}^{\infty} w_{-}(t) \; dt \geqslant C_{-}(\eta) T 
\end{equation}
and
\begin{equation} \label{w-plus--integral-ineq}
\int_{-\infty}^{\infty} w_{+}(t) \; dt \leqslant C_{+}(\eta) T, 
\end{equation}
for some (fixed) constants $C_{\pm}(\eta),$ which can be chosen as close to $1$ as we like. 

Summarizing, we conclude that IF 
\begin{align} \label{olikhet-sixth-kappa}
A_{\kappa} > \Big(\frac{\kappa}{\pi}\Big)^{2} 
\cdot \Big\{ &(v/2 - u)^{2} A_{\kappa} 
+ (v - 2u) B_{\kappa} 
+ C_{\kappa} + 
(v - 2u) D_{\kappa} 
+ (v - 2u)E_{\kappa} \nonumber \\
&+ F_{\kappa} 
+ G_{\kappa} 
+ 2 H_{\kappa} 
+ 2 I_{\kappa} 
+ 2 J_{\kappa} \Big\}, 
\end{align}
then we have a contradiction to our Main Assumption. That would imply the existence of a subinterval of $[T,2T - \frac{\kappa}{\log{T}}]$ of length at least $\frac{\kappa}{\log{T}},$ in which the function 
$t \mapsto P(t,u,v,\kappa)$ has no zeros. A simple proof by contradiction shows that this implies that there must be a subinterval of $[T,2T]$ of length at least $\frac{2\kappa}{\log{T}},$ in which the function 
$t \mapsto \zeta(\tfrac{1}{2} + it)$ has no zeros. 

Using Mathematica, the inequality (\ref{olikhet-sixth-kappa}) is seen to hold with $u = 0.0909,$ $v = 2.13$ and $\kappa = 8.69.$ Thus Theorem \ref{huvudsats} holds since $2 \cdot 8.69 > 2.766 \cdot 2\pi.$

\begin{remark} \label{hall-remark} 
If one studies the $\kappa$-inequality (\ref{olikhet-sixth-kappa}) as $u \to 0,$ one sees that (\ref{olikhet-sixth-kappa}) is satisfied for $\kappa = 8.264$ 
(with $v = 2$), yielding gaps of length at least $2.63$ times the average. This is effectively what Hall did in \cite{hall-fyra} (he did not use any amplifier).
\end{remark} 

\begin{remark} \label{u-values-remark} 
For what it is worth, note that if the results (again via \cite{HY}) would remain valid for any $u < 1/2$ (whether this is the case or not is unknown), then one could take $u = 0.4999,$ $v = 2.68$ and $\kappa = 10.23$ and see that (\ref{olikhet-sixth-kappa}) holds. This would imply the existence of gaps of length at least $3.25$ times the average. Moreover, $u = 0.55$ and $v = 2.74$ would yield gaps of length at least $3.26$ times the average and 
$u = 0.9999$ and $v = 3$ would yield gaps of length at least $3.05$ times the average\footnote{Being unable to explain why using $u = 1/2$ leads to bigger gaps than $u = 1,$ let me just mention that this was also the case when I (admittedly on rough paper and using ratios conjectures) looked at how amplifying the second moment of the Riemann zeta-function improved Hall's method for finding large gaps.}. 
\end{remark} 

\begin{remark} \label{improvement-remark} 
As a side-note, it is likely that replacing $M(\frac{1}{2} + it)$ in the definition of our function $P(t,u,v,\kappa)$ in (\ref{f-def-hall-ekv}) by 
\begin{equation*} \label{improvement-ide}
\sum_{h \leqslant T^{u}} 
\frac{A + B \frac{\log(T^{u}/h)}{\log(T^{u})}}{h^{1/2 + it}},
\end{equation*}
with some suitable choice of $A$ and $B,$ would have lead to a slightly better gap-result. However, such calculations would be very long. 
\end{remark} 

\section{Evaluation of our integrals} \label{section-evaluation-integrals}

\subsection{On the article ``The twisted fourth moment of the Riemann zeta function"} 
\label{subsection-on-the-article}

We will make use of the main theorem in the article ``The twisted fourth moment of the Riemann zeta function", written by Hughes and Young \cite{HY}. Before we reproduce their result, we must introduce a little bit of notation. 

Define 
\begin{equation} \label{A-def}
A_{\alpha,\beta,\gamma,\delta}(s) = \frac{\zeta(1 + s + \alpha + \gamma) 
\zeta(1 + s + \alpha + \delta) \zeta(1 + s + \beta + \gamma) 
\zeta(1 + s + \beta + \delta) }
{\zeta(2 + 2s + \alpha + \beta + \gamma + \delta)}.
\end{equation}
Let 
\begin{equation} \label{div-def}
\sigma_{\alpha,\beta}(n) = \sum_{n_{1}n_{2} = n} 
n_{1}^{-\alpha} n_{2}^{-\beta}.
\end{equation}
Next, suppose $(h,k) = 1,$ $p^{h_{p}}||h$ and $p^{k_{p}}||k,$ and define 
\begin{align} \label{B-euler-def}
B_{\alpha,\beta,\gamma,\delta,h,k}(s) &= \displaystyle\prod_{p|h} 
\bigg( \frac{ \sum_{j=0}^{\infty} \sigma_{\alpha,\beta}(p^{j}) 
\sigma_{\gamma,\delta}(p^{j + h_{p}}) p^{-j(s+1)} }
{\sum_{j=0}^{\infty} \sigma_{\alpha,\beta}(p^{j}) 
\sigma_{\gamma,\delta}(p^{j}) p^{-j(s+1)}} \bigg) \\
& \qquad \qquad \qquad \times \displaystyle\prod_{p|k} 
\bigg( \frac{ \sum_{j=0}^{\infty} \sigma_{\alpha,\beta}(p^{j + k_{p}}) 
\sigma_{\gamma,\delta}(p^{j}) p^{-j(s+1)} }
{\sum_{j=0}^{\infty} \sigma_{\alpha,\beta}(p^{j}) 
\sigma_{\gamma,\delta}(p^{j}) p^{-j(s+1)}} \bigg). \nonumber
\end{align}
Then we write
\begin{equation} \label{Z-sixth-def}
Z_{\alpha,\beta,\gamma,\delta,h,k}(s) 
= A_{\alpha,\beta,\gamma,\delta}(s) 
B_{\alpha,\beta,\gamma,\delta,h,k}(s).
\end{equation}

\begin{theorem}[Main theorem in \cite{HY}] \label{sats1}
Let 
\begin{equation} \label{I-art-def}
I(h,k) = \int_{-\infty}^{\infty} \Big( \frac{h}{k} \Big)^{-it} 
\zeta(\tfrac{1}{2} + \alpha + it) \zeta(\tfrac{1}{2} + \beta + it) 
\zeta(\tfrac{1}{2} + \gamma - it) \zeta(\tfrac{1}{2} + \delta - it) 
w(t) \; dt,
\end{equation}
where $w(t)$ is a smooth, non-negative function with support contained in 
$[\frac{T}{2}, 4T],$ satisfying $w^{(j)}(t) \ll_{j} T_{0}^{-j}$ for all 
$j = 0, 1, 2,...,$ where $T^{\frac{1}{2} + \epsilon} \ll T_{0} \ll T.$ Suppose $(h,k) = 1,$ $hk~\leqslant~T^{\frac{2}{11} - \epsilon},$ and that $\alpha, \beta, \gamma, \delta$ are complex numbers $\ll (\log{T})^{-1}.$ Then 
\begin{align} \label{sats1-ekv}
I(h,k) &= \frac{1}{\sqrt{hk}} \int_{-\infty}^{\infty} w(t) 
\bigg( Z_{\alpha,\beta,\gamma,\delta,h,k}(0) 
+ \Big( \frac{t}{2 \pi} \Big)^{-\alpha - \beta - \gamma - \delta} Z_{-\gamma,-\delta,-\alpha,-\beta,h,k}(0) \\
&+ \Big( \frac{t}{2 \pi} \Big)^{-\alpha - \gamma} Z_{-\gamma,\beta,-\alpha,\delta,h,k}(0) + 
\Big( \frac{t}{2 \pi} \Big)^{-\alpha - \delta} Z_{-\delta,\beta,\gamma,-\alpha,h,k}(0) \nonumber \\
&+ \Big( \frac{t}{2 \pi} \Big)^{-\beta - \gamma} Z_{\alpha,-\gamma,-\beta,\delta,h,k}(0) + 
\Big( \frac{t}{2 \pi} \Big)^{-\beta - \delta} Z_{\alpha,-\delta,\gamma,-\beta,h,k}(0) \bigg) \; dt \nonumber \\
&+ O(T^{\frac{3}{4} + \epsilon} (hk)^{\frac{7}{8}} (T/T_{0})^{\frac{9}{4}}). \nonumber 
\end{align}
\end{theorem}

\textbf{Brief comment on the proof:} The proof is very complicated and the reader is referred to \cite{HY}. What follows is just a very brief outline.

The starting point is an approximate functional equation. 
Let $G(s)$ be an even, entire function of rapid decay as $|s| \to \infty$ in any fixed strip $|\textrm{Re}(s)| \leqslant C$ and let 
\begin{equation} \label{V-AFE-def}
V_{\alpha,\beta,\gamma,\delta,t}(x) = \frac{1}{2 \pi i} 
\int_{(1)} \frac{G(s)}{s} g_{\alpha,\beta,\gamma,\delta}(s,t) x^{-s} \; ds,
\end{equation}
where
\begin{equation} \label{g-AFE-def}
g_{\alpha,\beta,\gamma,\delta}(s,t) 
= \frac{\Gamma\Big( \frac{\frac{1}{2} + \alpha + s + it}{2} \Big)}
{\Gamma\Big( \frac{\frac{1}{2} + \alpha + it}{2} \Big)} 
\frac{\Gamma\Big( \frac{\frac{1}{2} + \beta + s + it}{2} \Big)}
{\Gamma\Big( \frac{\frac{1}{2} + \beta + it}{2} \Big)} 
\frac{\Gamma\Big( \frac{\frac{1}{2} + \gamma + s - it}{2} \Big)}
{\Gamma\Big( \frac{\frac{1}{2} + \gamma - it}{2} \Big)} 
\frac{\Gamma\Big( \frac{\frac{1}{2} + \delta + s - it}{2} \Big)}
{\Gamma\Big( \frac{\frac{1}{2} + \delta - it}{2} \Big)}. 
\end{equation}
Then
\begin{align} \label{zetas-AFE}
&\zeta(\tfrac{1}{2} + \alpha + it) \zeta(\tfrac{1}{2} + \beta + it) 
\zeta(\tfrac{1}{2} + \gamma - it) \zeta(\tfrac{1}{2} + \delta - it) \\
&= \sum_{m,n \geqslant 1} 
\frac{\sigma_{\alpha,\beta}(m)\sigma_{\gamma,\delta}(n)}{(mn)^{1/2}} 
\Big( \frac{m}{n} \Big)^{-it} 
V_{\alpha,\beta,\gamma,\delta,t}(\pi^{2} mn) \nonumber \\
&\qquad \qquad + X_{\alpha,\beta,\gamma,\delta,t} \sum_{m,n \geqslant 1} 
\frac{\sigma_{-\gamma,-\delta}(m)\sigma_{-\alpha,-\beta}(n)}{(mn)^{1/2}} 
\Big( \frac{m}{n} \Big)^{-it} 
V_{-\gamma,-\delta,-\alpha,-\beta,t}(\pi^{2} mn) \nonumber \\
&\qquad \qquad + O((1 + |t|)^{-2007}), \nonumber
\end{align}
where
\begin{equation} \label{X-AFE-def}
X_{\alpha,\beta,\gamma,\delta,t} :=
\pi^{\alpha + \beta + \gamma + \delta} 
\frac{\Gamma\Big( \frac{\frac{1}{2} - \alpha - it}{2} \Big)}
{\Gamma\Big( \frac{\frac{1}{2} + \alpha + it}{2} \Big)} 
\frac{\Gamma\Big( \frac{\frac{1}{2} - \beta - it}{2} \Big)}
{\Gamma\Big( \frac{\frac{1}{2} + \beta + it}{2} \Big)} 
\frac{\Gamma\Big( \frac{\frac{1}{2} - \gamma + it}{2} \Big)}
{\Gamma\Big( \frac{\frac{1}{2} + \gamma - it}{2} \Big)} 
\frac{\Gamma\Big( \frac{\frac{1}{2} - \delta + it}{2} \Big)}
{\Gamma\Big( \frac{\frac{1}{2} + \delta - it}{2} \Big)}. 
\end{equation}

Using the approximate functional equation (\ref{zetas-AFE}) in the definition of $I(h,k)$ (see (\ref{I-art-def})) yields
\begin{align} \label{I-formel}
I(h,k) &= \sum_{m,n \geqslant 1} 
\frac{\sigma_{\alpha,\beta}(m)\sigma_{\gamma,\delta}(n)}{(mn)^{1/2}} 
\int_{-\infty}^{\infty} \Big( \frac{hm}{kn} \Big)^{-it} 
V_{\alpha,\beta,\gamma,\delta,t}(\pi^{2} mn) w(t) \; dt \\
& + \sum_{m,n \geqslant 1} \Big\{
\frac{\sigma_{-\gamma,-\delta}(m)\sigma_{-\alpha,-\beta}(n)}{(mn)^{1/2}} \nonumber \\
&\times \int_{-\infty}^{\infty} \Big( \frac{hm}{kn} \Big)^{-it} 
X_{\alpha,\beta,\gamma,\delta,t} 
V_{-\gamma,-\delta,-\alpha,-\beta,t}(\pi^{2} mn) w(t) \; dt \Big\} + O(1). \nonumber
\end{align}
The two main terms in (\ref{I-formel}) can be treated similarly. Denoting the first one by $I^{(1)}(h,k),$ then upon opening up the integral formula for $V,$ one finds that 
\begin{align} \label{I-ett-ekv}
&I^{(1)}(h,k) = \sum_{m,n \geqslant 1} \Big\{
\frac{\sigma_{\alpha,\beta}(m)\sigma_{\gamma,\delta}(n)}{(mn)^{1/2}} \\
&\times \frac{1}{2 \pi i} \int_{(1)} \frac{G(s)}{s} (\pi^{2} m n)^{-s}
\int_{-\infty}^{\infty} \Big( \frac{hm}{kn} \Big)^{-it} 
g_{\alpha,\beta,\gamma,\delta}(s,t) w(t) \; dt \; ds \Big\}. \nonumber
\end{align}

The authors of \cite{HY} split the sum in (\ref{I-ett-ekv}) into the diagonal part corresponding to the terms for which $hm = kn$ and the non-diagonal part (the other terms). Whereas it is relatively simple to treat the diagonal contribution, it was a nice achievement to be able to treat the more complicated off-diagonal terms (results from the article \cite{DFI} by Duke, Friedlander and Iwaniec are used and the latter part of the proof of Theorem \ref{sats1} involves a lot of simplifying).

\begin{remark} \label{T-noll-val-remark}
Throughout this article we will always use the choice 
$T_{0} = T^{1 - \epsilon}$ and the practice of letting $\epsilon$ stand for a small positive number, not necessarily always the same.
\end{remark}

\begin{remark} \label{analyticity-remark}
It is not immediately obvious that the main term in (\ref{sats1-ekv}) is an analytic function in terms of the shifts (e.g.\ 
$Z_{\alpha,\beta,\gamma,\delta}(0)$ has singularities at 
$\alpha = -\gamma,$ $\alpha = -\delta,$ $\beta = -\gamma,$ $\beta = -\delta$), however, due to nice cancellation, analyticity holds. Lemma 2.5.1 in the article \cite{CFKRS} by Conrey, Farmer, Keating, Rubinstein and Snaith is very helpful when showing this. In Section \ref{subsection-studying-Q} we will carry out the details in a similar situation.
\end{remark}

\subsection{Initial step in using Theorem \ref{sats1}} 
\label{subsection-initial-step}

Suppose that $w(t)$ is a function satisfying the assumptions in Theorem \ref{sats1}, given the choice $T_{0} = T^{1-\epsilon}.$ 
With
\begin{equation} \label{m-def-ett-kap-tva}
M_{1}(s) = \sum_{h \leqslant T^{u}} \frac{a_{1}(h)}{h^{s}}
\end{equation}
and
\begin{equation} \label{m-def-tva-kap-tva}
M_{2}(s) = \sum_{k \leqslant T^{u}} \frac{a_{2}(k)}{k^{s}}, 
\end{equation}
where $0 < u < 1/11$ and the $a_{i}$-coefficients are real, we will want to asymptotically evaluate expressions such as 
\begin{equation} \label{integral-bra}
\int_{-\infty}^{\infty} M_{1}(\tfrac{1}{2} + it) M_{2}(\tfrac{1}{2} - it) 
\zeta(\tfrac{1}{2} + \alpha + it) \zeta(\tfrac{1}{2} + \beta + it) 
\zeta(\tfrac{1}{2} + \gamma - it) \zeta(\tfrac{1}{2} + \delta - it) 
w(t) \; dt.
\end{equation}
By expanding out $M_{1}$ and $M_{2}$ we obtain 
\begin{equation} \label{ms-expanded-ett}
\sum_{h,k \leqslant T^{u}} \frac{a_{1}(h)a_{2}(k)}{\sqrt{hk}} 
\int_{-\infty}^{\infty} \Big( \frac{h}{k} \Big)^{-it}
\zeta(\tfrac{1}{2} + \alpha + it) \zeta(\tfrac{1}{2} + \beta + it) 
\zeta(\tfrac{1}{2} + \gamma - it) \zeta(\tfrac{1}{2} + \delta - it) 
w(t) \; dt. 
\end{equation}
Let us write (\ref{sats1-ekv}) as 
\begin{equation} \label{I-J-E}
I(h,k) = J(h,k) + E(h,k). 
\end{equation}  
Then (\ref{ms-expanded-ett}) equals 
\begin{align} \label{ms-expanded-tva}
&\sum_{h,k \leqslant T^{u}} \frac{a_{1}(h)a_{2}(k)}{\sqrt{hk}} 
I(h,k) \nonumber \\
& = \sum_{m \leqslant T^{u}} \frac{1}{m} 
\sum_{\substack{h,k \leqslant T^{u}/m \\ (h,k) = 1}} 
\frac{a_{1}(hm)a_{2}(km)}{\sqrt{hk}} I(h,k) \nonumber \\
& = \sum_{m \leqslant T^{u}} \frac{1}{m} 
\sum_{\substack{h,k \leqslant T^{u}/m \\ (h,k) = 1}} 
\frac{a_{1}(hm)a_{2}(km)}{\sqrt{hk}} J(h,k) 
+ \sum_{m \leqslant T^{u}} \frac{1}{m} 
\sum_{\substack{h,k \leqslant T^{u}/m \\ (h,k) = 1}} 
\frac{a_{1}(hm)a_{2}(km)}{\sqrt{hk}} E(h,k).
\end{align}
The first term in (\ref{ms-expanded-tva}) clearly equals 
\begin{align} \label{huvuddel-summor}
&\sum_{m \leqslant T^{u}} \frac{1}{m} 
\sum_{h,k \leqslant T^{u}/m} 
\frac{a_{1}(hm)a_{2}(km)}{\sqrt{hk}} J(h,k)
\sum_{d | (h,k)} \mu(d) \nonumber \\
& = \sum_{m \leqslant T^{u}} \frac{1}{m} 
\sum_{d \leqslant T^{u}/m} \frac{\mu(d)}{d}
\sum_{h,k \leqslant T^{u}/md} 
\frac{a_{1}(hmd)a_{2}(kmd)}{\sqrt{hk}} J(hd,kd)
\nonumber \\
& = \sum_{d \leqslant T^{u}} \frac{\mu(d)}{d} 
\sum_{m \leqslant T^{u}/d} \frac{1}{m}
\sum_{h,k \leqslant T^{u}/md} 
\frac{a_{1}(hmd)a_{2}(kmd)}{\sqrt{hk}} J(hd,kd). 
\end{align}
The second term in (\ref{ms-expanded-tva}) is 
\begin{equation} \label{feldel-summor}
\ll T^{\frac{3}{4} + \epsilon} (T/T_{0})^{\frac{9}{4}} 
\sum_{m \leqslant T^{u}} \Big\{ \frac{1}{m} 
\sum_{h \leqslant T^{u}/m} |a_{1}(hm)| h^{3/8} 
\sum_{k \leqslant T^{u}/m} |a_{2}(km)| k^{3/8} \Big\}. 
\end{equation}

\subsection{Specialising on the most standard case} 
\label{subsection-specialising}

Let us now investigate the expression in (\ref{integral-bra}) when 
$a_{1}(h) = a_{2}(k) = 1$ for all values of $h$ and $k.$ Our goal is to simplify the main term (which will turn out to be of order $T \log^{9}{T}$) as much as possible, treating anything which is $\ll T\log^{8}{T}$ as an error term.

As noticed in Section \ref{subsection-initial-step}, (\ref{integral-bra}) splits up into a main term (\ref{huvuddel-summor}) and an error term (\ref{feldel-summor}). The latter is 
\begin{equation*} \label{feluppskattning-2.2}
\ll  T^{\frac{3}{4} + \epsilon} (T/T_{0})^{\frac{9}{4}} 
\sum_{m \leqslant T^{u}} \Big\{ \frac{1}{m} \cdot (T^{u}/m)^{11/8} 
\cdot (T^{u}/m)^{11/8} \Big\} 
\ll T^{\frac{3}{4} + \epsilon} (T)^{\frac{9 \epsilon}{4}} 
T^{\frac{11 u}{4}} \ll T, 
\end{equation*}
recalling Remark \ref{T-noll-val-remark} for the second step and the last step being true upon choosing $\epsilon$ to be sufficiently small. 
The main term (\ref{huvuddel-summor}) is 
\begin{align} \label{main-term-ett-borjan}
&\sum_{d \leqslant T^{u}} \frac{\mu(d)}{d} 
\sum_{m \leqslant T^{u}/d} \frac{1}{m}
\sum_{h,k \leqslant T^{u}/md} 
\frac{J(hd,kd)}{\sqrt{hk}} \nonumber \\
&= \sum_{d \leqslant T^{u}} \frac{\mu(d)}{d^{2}} 
\sum_{m \leqslant T^{u}/d} \frac{1}{m}
\int_{-\infty}^{\infty} w(t) 
\sum_{h,k \leqslant T^{u}/md} 
\frac{1}{hk} \Big\{ ... \Big\} \; dt,
\end{align}
where the expression $\{ ... \}$ in (\ref{main-term-ett-borjan}) stands for 
\begin{align} \label{Hy-integrand-uttrycket}
&Z_{\alpha,\beta,\gamma,\delta,hd,kd}(0) 
+ \Big( \frac{t}{2 \pi} \Big)^{-\alpha - \beta - \gamma - \delta} Z_{-\gamma,-\delta,-\alpha,-\beta,hd,kd}(0) \nonumber \\ 
&\qquad + \Big( \frac{t}{2 \pi} \Big)^{-\alpha - \gamma} Z_{-\gamma,\beta,-\alpha,\delta,hd,kd}(0) 
+ \Big( \frac{t}{2 \pi} \Big)^{-\alpha - \delta} Z_{-\delta,\beta,\gamma,-\alpha,hd,kd}(0) \nonumber \\
&\qquad \qquad + \Big( \frac{t}{2 \pi} \Big)^{-\beta - \gamma} Z_{\alpha,-\gamma,-\beta,\delta,hd,kd}(0) + 
\Big( \frac{t}{2 \pi} \Big)^{-\beta - \delta} Z_{\alpha,-\delta,\gamma,-\beta,hd,kd}(0).  
\end{align}

\subsection{Studying $\mathcal{Q}_{A,B}(T_{1},T_{2},f_{1},f_{2})$} 
\label{subsection-studying-Q}

As we shall see in the Sections \ref{subsection-initial-simplification} and \ref{subsection-further-simplification}, it is possible to simplify our main term (\ref{main-term-ett-borjan}) considerably. However, we first need to introduce and become familiar with a bit of new notation. 

Suppose that $f_{1}(x_{1},x_{2},T_2)$ and $f_{2}(x_{1},x_{2},T_2)$ are functions which are analytic and symmetric in their complex variables 
$x_{1}$ and $x_{2}.$ Also, let $A$ and $B$ be sets of complex numbers with 
$|A| = |B| = 2$ and write them as 
\begin{equation} \label{A-set-intro}
A := \{ \alpha_{1}, \alpha_{2} \} 
\end{equation} 
and
\begin{equation} \label{B-set-intro}
B := \{ \alpha_{3}, \alpha_{4} \}. 
\end{equation} 
We then define 
\begin{equation} \label{q-curly-stor}
\mathcal{Q}_{A,B}(T_{1},T_{2},f_{1},f_{2}) 
:= \displaystyle\sum_{\substack{R \subseteq A \\ S \subseteq B \\ |R| = |S|}} 
\mathcal{Q}((A \setminus R) \cup (-S), (B \setminus S) \cup(-R), 
T_{1}, T_{2}, f_{1}, f_{2}), 
\end{equation}
where we by $-U$ mean $\{ -u : u \in U \}$ and 
\begin{equation} \label{Q-fyra-def}
\mathcal{Q}(X,Y,T_{1},T_{2},f_{1},f_{2}) 
:= T_{1}^{(\delta_{X} + \delta_{Y})/2} 
F_{1}(X,T_{2}) F_{2}(Y,T_{2}) 
\displaystyle\prod_{\substack{x \in X \\ y \in Y}} 
\frac{1}{(x + y)}, 
\end{equation}
with 
\begin{equation} \label{delta-sets-def}
\delta_{X} := \sum_{x \in X} x
\end{equation}
and where for $X = \{ x_{1}, x_{2} \}$ we let 
\begin{equation} \label{Ff-defs}
F_{i}(X,T_{2}) := f_{i}(x_{1},x_{2},T_{2}), \; i = 1,2.
\end{equation}

Next let $\Xi$ denote the set of $\binom{4}{2}$ permutations $\sigma \in S_{4}$ satisfying $\sigma(1) < \sigma(2)$ and $\sigma(3) < \sigma(4).$ Let us for 
$\sigma \in \Xi$ define 
\begin{align} \label{k-def-lemmat}
K(\alpha_{\sigma(1)},\alpha_{\sigma(2)};\alpha_{\sigma(3)},\alpha_{\sigma(4)}) 
&= K(\alpha_{\sigma(1)}, \alpha_{\sigma(2)}; \alpha_{\sigma(3)}, \alpha_{\sigma(4)}, T_{1}, T_{2}, f_{1}, f_{2}) \nonumber \\
&:= \mathcal{Q}(\{ \alpha_{\sigma(1)}, \alpha_{\sigma(2)} \}, 
\{ -\alpha_{\sigma(3)}, -\alpha_{\sigma(4)} \}, T_{1}, T_{2}, f_{1}, f_{2}). \qquad \qquad \quad \nonumber
\end{align}
Then one has that 
\begin{equation} \label{k-Q-samband}
\sum_{\sigma \in \Xi} 
K(\alpha_{\sigma(1)},\alpha_{\sigma(2)};\alpha_{\sigma(3)},\alpha_{\sigma(4)}) 
= \mathcal{Q}_{A,-B}(T_{1},T_{2},f_{1},f_{2}).
\end{equation}

We now show that the LHS of (\ref{k-Q-samband}) is an analytic function of the $\alpha$-shifts. The problem is where any of the relevant $K(\alpha_{\sigma(1)},\alpha_{\sigma(2)};\alpha_{\sigma(3)},\alpha_{\sigma(4)})$ has a singularity. However, we will now show that these singularities must be removable. Suppose that 
\begin{equation} \label{cfkrs-villkorrrr}
\alpha_i \neq \alpha_j, \textrm{for } i \neq j.
\end{equation} 
From Lemma 2.5.1 in \cite{CFKRS} we then have the following formula: 
\begin{align} \label{cfkrs-reprrrr}
&\sum_{\sigma \in \Xi} K(\alpha_{\sigma(1)},\alpha_{\sigma(2)}; \alpha_{\sigma(3)},\alpha_{\sigma(4)}) \nonumber \\
&= \frac{1}{(2!)^2} \frac{1}{(2\pi i)^4} \oint \cdots \oint 
\frac{K(z_1,z_2;z_3,z_4)\Delta(z_1,...,z_4)^2}
{\prod_{i=1}^{4} \prod_{j=1}^{4} (z_i - \alpha_j)} \, dz_1 ... dz_4,
\end{align} 
where
\begin{equation*}
\Delta(z_1,...,z_4) := \displaystyle\prod_{1 \leqslant i < j \leqslant 4} (z_j - z_i),
\end{equation*}
and where one integrates about circles enclosing the $\alpha_j$'s.
By choosing the radii of the circles to be suitably large\footnote{If analyticity is to be shown for say $|\alpha_{i}| \leqslant C,$ then we can pick the radii of the circles to be $3C$ since then 
$|z_{i} - \alpha_{j}|^{-1} \leqslant 1/C.$}, 
we obtain an upper bound for the RHS of (\ref{cfkrs-reprrrr}). 
The function $\sum_{\sigma \in \Xi} K(\alpha_{\sigma(1)},\alpha_{\sigma(2)}; \alpha_{\sigma(3)},\alpha_{\sigma(4)})$ thus remains bounded whenever (\ref{cfkrs-villkorrrr}) is satisfied. This allows us to conclude\footnote{By Riemann's Extension Theorem --- see for example \cite{GR-complex-an}.} that the possible singularities must be removable. Hence $\mathcal{Q}_{A,-B}(T_{1},T_{2},f_{1},f_{2})$ is analytic by (\ref{k-Q-samband}), which in turn implies that 
$\mathcal{Q}_{A,B}(T_{1},T_{2},f_{1},f_{2})$ is an analytic function of the shifts. 

\subsection{Initial simplification of Theorem \ref{sats1} in the most standard case} 
\label{subsection-initial-simplification}

\begin{theorem} \label{simplified-enkla-fallet}
Suppose that $w(t)$ is a smooth, non-negative function with support contained in 
$[\frac{T}{2}, 4T]$ satisfying $w^{(j)}(t) \ll_{j} (T^{1-\epsilon})^{-j}$ for all $j = 0, 1, 2,...,$ that $\alpha, \beta, \gamma, \delta$ are complex numbers 
$\ll (\log{T})^{-1}$ and let 
\begin{equation*}
M(s) = \sum_{h \leqslant T^{u}} \frac{1}{h^{s}},
\end{equation*}
with $0 < u < 1/11.$ Then 
\begin{align} \label{simplified-enkla-fallet-ekv}
&\int_{-\infty}^{\infty} |M(\tfrac{1}{2} + it)|^{2} 
\zeta(\tfrac{1}{2} + \alpha + it) \zeta(\tfrac{1}{2} + \beta + it) 
\zeta(\tfrac{1}{2} + \gamma - it) \zeta(\tfrac{1}{2} + \delta - it) 
w(t) \; dt \nonumber \\
&= \Big[ \int_{-\infty}^{\infty} w(t) \; dt \Big] 
\cdot a_{3} 
\cdot T^{-(\alpha + \beta + \gamma + \delta)/2} 
\cdot \sum_{m \leqslant T^{u}} \frac{1}{m} 
\mathcal{Q}_{A,B}(T,T^{u}/m,f,f) + O(T \log^{8}{T}), 
\end{align}
with 
\begin{equation*} 
a_{3} = \displaystyle\prod_{p} 
\Big\{ \Big(1 + \frac{4}{p} + \frac{1}{p^{2}}\Big)
\Big(1 - \frac{1}{p}\Big)^{4} \Big\}
\end{equation*}
and where we are here taking (see also definition given in (\ref{Q-fyra-def}))
\begin{equation} \label{A-val-set}
A := \{\alpha, \beta\}, 
\end{equation}
\begin{equation} \label{B-val-set}
B := \{\gamma, \delta\}
\end{equation}
and\footnote{When comparing (\ref{A-val-set}) with 
(\ref{A-set-intro}), the reader will probably find it easiest to just identify $\alpha$ and $\beta$ with $\alpha_{1}$ and $\alpha_{2}$ respectively. Note however that due to symmetry, it does not matter if the roles are reversed.} 
\begin{equation} \label{f-function-spec-ett}
f(x_{1},x_{2},T_{2}) 
:= \frac{1}{x_{1} x_{2}} - \frac{T_{2}^{-x_{1}}}{x_{1}(x_{2} - x_{1})} 
- \frac{T_{2}^{-x_{2}}}{x_{2}(x_{1} - x_{2})}. 
\end{equation}
\end{theorem}

\begin{remark} \label{anal-of-f-standard}
It is easy to show that $f(x_{1},x_{2},T_{2})$ is analytic in 
$x_{1}$ and $x_{2}.$ Assume (without loss of generality) that 
$|x_{1}|, |x_{2}| \leqslant C$ say. Consider 
\begin{equation} \label{analytic-integral-enkelt-ju} 
\frac{1}{2 \pi i} \int_{R} \frac{T_{2}^{s}}{s(s+x_{1})(s+x_{2})} \; ds, 
\end{equation} 
with $R$ denoting a counter-clockwise integral-contour around a square with vertices at 
$\pm 2C \pm 2Ci.$ Suppose that $x_{1}$ and $x_{2}$ are different and non-zero. Then (\ref{analytic-integral-enkelt-ju}) equals (\ref{f-function-spec-ett}) by Cauchy's Residue Theorem, and the integral in (\ref{analytic-integral-enkelt-ju}) is obviously bounded. It follows that any possible singularities of $f(x_{1},x_{2},T_{2})$ must be removable. 
\end{remark}

\begin{remark} \label{analyticity-HL-remark}
By Section \ref{subsection-studying-Q} we thus know that $\mathcal{Q}_{A,B}(T_{1},T_{2},f,f)$ is analytic in terms of the shifts. Hence the main term in the RHS of (\ref{simplified-enkla-fallet-ekv}) is analytic in the shifts.
\end{remark}

\begin{proof}
We will first proceed under the assumption that for some 
fixed constant $C > 0$ we have that 
\begin{equation} \label{extra-condition}
 |\alpha_i| \geqslant C/\log{T}, \; \; \; |\alpha_i + \alpha_j| \geqslant C/\log{T} \; \; \; \; \textrm{and} \; \; \; |\alpha_i - \alpha_j| \geqslant C/\log{T}, 
\end{equation}
where $\alpha_{i}$ and $\alpha_{j}$ (with $i$ and $j$ distinct) stand for any of the shifts.

Let us begin by noticing that there is an obvious identification to be made between the terms in (\ref{main-term-ett-borjan}) and in the main term of (\ref{simplified-enkla-fallet-ekv}) (indeed both expressions involve the same number of terms, namely six)\footnote{The reader may find it helpful to factor out $( \frac{t}{2 \pi} )^{-(\alpha + \beta + \gamma + \delta)/2}$ in (\ref{Hy-integrand-uttrycket}).}. We shall prove Theorem \ref{simplified-enkla-fallet} under the assumption (\ref{extra-condition}) by treating each term in (\ref{main-term-ett-borjan}) individually. Below we will focus on the third term in (\ref{main-term-ett-borjan}), which will correspond (recalling (\ref{q-curly-stor})) to the case $R = \{ \alpha \}$ and $S = \{ \gamma \}$ in (\ref{simplified-enkla-fallet-ekv}). The other terms can be treated analogously.

Our starting point will thus be given by
\begin{equation} \label{tredje-termen-i-bevis}
\Big\{ \int_{-\infty}^{\infty} w(t) 
\Big( \frac{t}{2 \pi} \Big)^{-\alpha - \gamma} \; dt \Big\}
\sum_{d \leqslant T^{u}} \frac{\mu(d)}{d^{2}} 
\sum_{m \leqslant T^{u}/d} \frac{1}{m}
\sum_{h,k \leqslant T^{u}/md} 
\frac{Z_{-\gamma,\beta,-\alpha,\delta,hd,kd}(0)}{hk}. 
\end{equation}
Let us write (recall (\ref{Z-sixth-def})) 
\begin{equation} \label{z-redef}
Z_{\alpha,\beta,\gamma,\delta,h,k}(0) 
= A_{\alpha,\beta,\gamma,\delta}(0) 
B_{\alpha,\beta,\gamma,\delta}(h)
E_{\alpha,\beta,\gamma,\delta}(k),
\end{equation}
with 
\begin{equation} \label{B-ny-euler-def}
B_{\alpha,\beta,\gamma,\delta}(h) := 
\displaystyle\prod_{p|h} 
\bigg( \frac{ \sum_{j=0}^{\infty} \sigma_{\alpha,\beta}(p^{j}) 
\sigma_{\gamma,\delta}(p^{j + h_{p}}) p^{-j}}
{\sum_{j=0}^{\infty} \sigma_{\alpha,\beta}(p^{j}) 
\sigma_{\gamma,\delta}(p^{j}) p^{-j}} \bigg)
\end{equation}  \\
and 
\begin{equation} \label{E-ny-euler-def}
E_{\alpha,\beta,\gamma,\delta}(k) 
:= \displaystyle\prod_{p|k} 
\bigg( \frac{ \sum_{j=0}^{\infty} \sigma_{\alpha,\beta}(p^{j + k_{p}}) 
\sigma_{\gamma,\delta}(p^{j}) p^{-j} }
{\sum_{j=0}^{\infty} \sigma_{\alpha,\beta}(p^{j}) 
\sigma_{\gamma,\delta}(p^{j}) p^{-j}} \bigg).
\end{equation}
Using this notation, (\ref{tredje-termen-i-bevis}) becomes 
\begin{align} \label{tredje-termen-i-bevis-II}
&\Big\{ \int_{-\infty}^{\infty} w(t) 
\Big( \frac{t}{2 \pi} \Big)^{-\alpha - \gamma} \; dt \Big\}
A_{-\gamma,\beta,-\alpha,\delta}(0) \nonumber \\
&\times \sum_{d \leqslant T^{u}} \frac{\mu(d)}{d^{2}} 
\sum_{m \leqslant T^{u}/d} 
\bigg\{ \frac{1}{m}
\sum_{h \leqslant T^{u}/md} 
\frac{B_{-\gamma,\beta,-\alpha,\delta}(hd)}{h}
\sum_{k \leqslant T^{u}/md} 
\frac{E_{-\gamma,\beta,-\alpha,\delta}(kd)}{k} \bigg\}.
\end{align}

It would be desirable to have more information about the two innermost sums in (\ref{tredje-termen-i-bevis-II}). They are similar. Let us study 
$\sum_{h \leqslant T^{u}/md} 
\frac{B_{\alpha,\beta,\gamma,\delta}(hd)}{h}$.
First note that 
\begin{equation} \label{B-euler-upper-bound}
B_{\alpha,\beta,\gamma,\delta}(w) 
= \displaystyle\prod_{p|w} B_{\alpha,\beta,\gamma,\delta}(p^{w_{p}}) 
\ll \displaystyle\prod_{p|w} w_{p} p^{\epsilon w_{p}} 
\ll \displaystyle\prod_{p|w} p^{\epsilon w_{p}} \cdot p^{\epsilon w_{p}} 
= w^{2 \epsilon}.
\end{equation}
Now define for Re$(s) > 1,$ 
\begin{equation} \label{F-fn-deff}
F(s,\alpha,\beta,\gamma,\delta,d) := \sum_{h = 1}^{\infty} 
\frac{B_{\alpha,\beta,\gamma,\delta}(hd)}{h^{s}}.
\end{equation}
Suppose for the time being that $\frac{T^{u}}{md}$ is a half-integer, say $\frac{T^{u}}{md} = M + \frac{1}{2}$ for some positive integer $M.$ We claim that then 
\begin{equation} \label{euler-evaluatiooon-II}
\sum_{h \leqslant T^{u}/md} 
\frac{B_{\alpha,\beta,\gamma,\delta}(hd)}{h} 
= \frac{1}{2 \pi i} \int_{\epsilon - iW}^{\epsilon + iW} 
\frac{F(1+s,\alpha,\beta,\gamma,\delta,d) (T^{u}/md)^{s}}{s} \; ds 
+ O(d^{\epsilon}),
\end{equation}
with (somewhat arbitrary choice) 
\begin{equation} \label{W-choice-int}
W = 10 \times (T^{u}/md)^{1.1}.
\end{equation}
It is easy to show (this is one version of Perron's formula) that 
\begin{equation} \label{perron}
\frac{1}{2 \pi i} \int_{\epsilon - iW}^{\epsilon + iW} \frac{x^{s}}{s} \; ds 
= H(x) + O\Big( \frac{x^{\epsilon}}{W|\log{x}|} \Big),
\end{equation}
for $x > 0,$ $x \neq 1,$ where
\begin{equation} \label{H-perron}
H(x)= \left\{ \begin{array}{l l}
  1 & \quad \mbox{if $x > 1$,}\\
  0 & \quad \mbox{if $x < 1.$}\\
\end{array} \right. 
\end{equation}
Expanding out $F(1+s,\alpha,\beta,\gamma,\delta,d)$ in (\ref{euler-evaluatiooon-II}), 
we thus get 
\begin{align} \label{perron-mellan}
&\frac{1}{2 \pi i} \int_{\epsilon - iW}^{\epsilon + iW} 
\frac{F(1+s,\alpha,\beta,\gamma,\delta,d) (T^{u}/md)^{s}}{s} \; ds 
\nonumber \\
&= \sum_{h \leqslant T^{u}/md} 
\frac{B_{\alpha,\beta,\gamma,\delta}(hd)}{h} 
+ O\Big( \sum_{h = 1}^{\infty} \frac{|B_{\alpha,\beta,\gamma,\delta}(hd)|}{h} 
\cdot \frac{(T^{u}/mdh)^{\epsilon}}{W |\log(T^{u}/mdh)|} \Big).
\end{align}
If $h \notin [ \frac{T^{u}}{2md}, \frac{3T^{u}}{2md} ],$  then $|\log(T^{u}/mdh)|^{-1} \ll 1.$ The part of the error in (\ref{perron-mellan}) corresponding to such values of $h$ is thus 
\begin{equation*} \label{perron-mellan-fel-ett}
\ll \sum_{h = 1}^{\infty} \frac{(hd)^{\epsilon / 2}}{h} 
\cdot \frac{(T^{u}/mdh)^{\epsilon}}{W} 
\ll d^{\epsilon}.
\end{equation*}
For $M + \frac{1}{2} = \frac{T^{u}}{md} < h \leqslant \frac{3T^{u}}{2md},$ we write 
\[ h = M + R, \qquad R = 1,...,\Big[ \frac{3T^{u}}{2md} \Big] - M\] 
and spot that here 
\begin{equation*} \label{perron-logfel}
|\log(T^{u}/mdh)|^{-1} \ll \frac{M}{R}, 
\end{equation*} 
so that the part of the error in (\ref{perron-mellan}) corresponding to these values of $h$ is certainly 
\begin{equation*} \label{perron-fel-tva}
\ll \frac{d^{\epsilon}}{W} \sum_{R = 1}^{M} \frac{M}{R} 
\ll \frac{d^{\epsilon}}{W} \cdot M \log{(M + 1)} \ll d^{\epsilon}.
\end{equation*}
A very similar argument applies when 
$\frac{T^{u}}{2md} \leqslant h \leqslant \frac{T^{u}}{md},$ which concludes the proof of the claim in the case when $\frac{T^{u}}{md}$ is a half-integer.
If this is not the case, then certainly for some $0 < \mu < 1$ we have that 
$\frac{T^{u}}{md} + \mu$ is a half-integer. We obtain\footnote{By looking at the derivation of (\ref{euler-evaluatiooon-II}), it is obvious that introducing $\mu$ does not necessitate a change in our choice of $W.$} 
\begin{align} \label{euler-evaluatiooon-III}
&\sum_{h \leqslant T^{u}/md} 
\frac{B_{\alpha,\beta,\gamma,\delta}(hd)}{h} 
= \sum_{h \leqslant T^{u}/md + \mu} 
\frac{B_{\alpha,\beta,\gamma,\delta}(hd)}{h}
+ O(d^{\epsilon}) \nonumber \\
&= \frac{1}{2 \pi i} \int_{\epsilon - iW}^{\epsilon + iW} 
\frac{F(1+s,\alpha,\beta,\gamma,\delta,d) (T^{u}/md + \mu)^{s}}{s} \; ds 
+ O(d^{\epsilon}).
\end{align}

We are thus lead to investigating the RHS of (\ref{euler-evaluatiooon-III}). By contemplating the definition of 
$B_{\alpha,\beta,\gamma,\delta}(hd),$ one concludes (see (\ref{B-ett-del})-(\ref{B-p-storre-del})) that one may write 
\begin{equation} \label{G-fn-deff}
F(s,\alpha,\beta,\gamma,\delta,d) =: 
\zeta(s + \gamma) \zeta(s + \delta) G(s,\alpha,\beta,\gamma,\delta,d),
\end{equation}
with $G(s,\alpha,\beta,\gamma,\delta,d)$ being an analytic function for 
Re$(s) > 1/2.$ Let us in the following discussion restrict ourselves to when 
$0.9 \leqslant$ Re$(s) \leqslant 1.1.$ We have
\begin{align*} 
G(s,\alpha,\beta,\gamma,\delta,d) &= 
\prod_{p | d} \big\{ (1-p^{-s-\gamma})(1-p^{-s-\delta}) \big\}
\displaystyle\sum_{\substack{h = 1 \\ p|h \Rightarrow p|d}}^{\infty} 
\frac{B_{\alpha,\beta,\gamma,\delta}(hd)}{h^s} \\
&\times \prod_{p \nmid d} \bigg\{ (1-p^{-s-\gamma})(1-p^{-s-\delta}) 
\displaystyle\sum_{M = 0}^{\infty} 
\frac{B_{\alpha,\beta,\gamma,\delta}(p^{M})}{p^{Ms}} \bigg\}.
\end{align*}
Let us write this as 
\begin{equation} \label{g-split} 
G(s,\alpha,\beta,\gamma,\delta,d) = 
G_1(s,\alpha,\beta,\gamma,\delta,d) \times G_2(s,\alpha,\beta,\gamma,\delta,d).
\end{equation}

We notice that (uniformly) 
\begin{equation} \label{B-ett-del}
B_{\alpha,\beta,\gamma,\delta}(1) = 1, 
\end{equation} 
\begin{equation} \label{B-p-del}
B_{\alpha,\beta,\gamma,\delta}(p) 
= p^{-\gamma} + p^{-\delta} + O(p^{-0.9 + 2 \epsilon}) 
\end{equation} 
and
\begin{equation} \label{B-p-storre-del}
B_{\alpha,\beta,\gamma,\delta}(p^N) = O((N + 1) p^{N \epsilon}), N \geqslant 2. 
\end{equation} 
Thus 
\[ \displaystyle\sum_{M = 0}^{\infty} 
\frac{B_{\alpha,\beta,\gamma,\delta}(p^M)}{p^{Ms}} 
= 1 + p^{-s-\gamma} + p^{-s-\delta} + O(p^{-1.8 + 2 \epsilon}) \]
and hence
\begin{equation*} \label{G-tva-result}
G_2(s,\alpha,\beta,\gamma,\delta,d) 
= \prod_{p \nmid d}(1 + O(p^{-1.8 + 2 \epsilon})).
\end{equation*}
In particular we obtain 
\begin{equation*} \label{G-tva-bound-tjaha}
G_2(s,\alpha,\beta,\gamma,\delta,d) \ll 1.
\end{equation*}
The main thing to notice in the above argument is that the terms of order (roughly) $p^{-\mathrm{Re}(s)}$ exactly cancel in the factors of 
$G_2(s,\alpha,\beta,\gamma,\delta,d)$. Keeping this in mind, one can show 
\begin{equation*} \label{g-tva-logs}
\frac{\partial G_2(s,\alpha,\beta,\gamma,\delta,d)}{\partial x} 
\ll \sum_{p \nmid d} \log p \cdot p^{-1.8 + 2\epsilon} \cdot 1 \ll 1,
\end{equation*}
where $x$ here shall mean any of $\alpha, \beta, \gamma, \delta, s.$

Now we study $G_1(s,\alpha,\beta,\gamma,\delta,d).$ First of all, 
\begin{equation*} 
G_{11}(s,\alpha,\beta,\gamma,\delta,d) 
:= \prod_{p | d} \{ (1-p^{-s-\gamma})(1-p^{-s-\delta}) \} 
\ll \prod_{p | d} \{ 2 \cdot 2 \} \ll  d^\epsilon. 
\end{equation*}
Also, logarithmic differentiation yields 
\begin{equation*} \label{g-elva-logs}
\frac{\partial G_{11}(s,\alpha,\beta,\gamma,\delta,d)}{\partial x} 
\ll d^\epsilon.
\end{equation*}
Next, we study 
\[ G_{12}(s,\alpha,\beta,\gamma,\delta,d) := 
\displaystyle\sum_{\substack{h = 1 \\ p|h \Rightarrow p|d}}^{\infty} 
\frac{B_{\alpha,\beta,\gamma,\delta}(hd)}{h^s}. \]
Recalling (\ref{B-euler-upper-bound}) we find 
\begin{align*} 
G_{12}(s,\alpha,\beta,\gamma,\delta,d) 
\ll d^\epsilon 
\displaystyle\sum_{\substack{h = 1 \\ p|h \Rightarrow p|d}}^{\infty} 
h^{-0.9 + \epsilon} 
&\ll d^\epsilon \displaystyle\prod_{p|d} \sum_{j=0}^{\infty}  
p^{j(-0.9 + \epsilon)} \\
&\ll d^\epsilon \displaystyle\prod_{p|d} (1 + O(p^{-0.9 + \epsilon})) 
\ll d^\epsilon. 
\end{align*}
Also, although the details are somewhat more delicate, one can in a straight-forward direct way do a similar upper bound calculation as above in order to deduce that 
\[ \frac{\partial {G_{12}(s,\alpha,\beta,\gamma,\delta,d)}}{\partial x} 
\ll d^\epsilon. \]

Putting things together we have 
\begin{equation} \label{g-bound}
G(s,\alpha,\beta,\gamma,\delta,d) \ll d^\epsilon
\end{equation}
and 
\begin{equation} \label{g-diff-bound}
\frac{\partial {G(s,\alpha,\beta,\gamma,\delta,d)}}{\partial x} 
\ll d^\epsilon.
\end{equation}

Now it is time to go back to (\ref{euler-evaluatiooon-III}) which tells us that 
\begin{align} \label{euler-evaluatiooon-IV}
&\sum_{h \leqslant T^{u}/md} 
\frac{B_{-\gamma,\beta,-\alpha,\delta}(hd)}{h} \nonumber \\
&= \frac{1}{2 \pi i} \int_{\epsilon - iW}^{\epsilon + iW} 
\frac{\zeta(1 - \alpha + s) \zeta(1 + \delta + s) 
(T^{u}/md + \mu)^{s} G(1+s,-\gamma,\beta,-\alpha,\delta,d)}{s} \; ds 
+ O(d^{\epsilon}).
\end{align}
Using a rectangular path, we move the line of integration from 
Re$(s) = \epsilon$ to Re$(s) = -0.05.$ The contribution along any of the two horizontal line-segments is  
\begin{equation*} \label{hor-int-bound}
\ll \int_{-0.05}^{\epsilon} \frac{(W^{0.1 \times \frac{1}{6} + \epsilon})^{2} 
(T^{u}/md + \mu)^{\epsilon} d^{\epsilon}}{W} \; dx 
\ll d^{\epsilon}.
\end{equation*}
And the contribution along the new vertical line-segment is 
\begin{align*} \label{horiz-ny-int-bound}
\int_{-0.05 - iW}^{-0.05 + iW} &= \int_{-0.05 - iW}^{-0.05 - 2 i} 
+ \int_{-0.05 - 2 i}^{-0.05 + 2 i} + \int_{-0.05 + 2i}^{-0.05 + iW} \nonumber \\ 
&\ll 2 \int_{2}^{W} \frac{(t^{0.1 \times \frac{1}{6} + \epsilon})^{2} 
(T^{u}/md + \mu)^{-0.05} d^{\epsilon}}{t} \; dt 
+ \int_{-2}^{2} \frac{(T^{u}/md + \mu)^{-0.05} d^{\epsilon}}{0.05} \; dt \nonumber \\ 
&\ll (T^{u}/md)^{-0.05} d^{\epsilon} 
\Big\{ \int_{2}^{W} t^{-\frac{29}{30} + \epsilon}  \; dt + 1 \Big\} \nonumber \\ 
& \ll d^{\epsilon}.
\end{align*}
By Cauchy's Residue Theorem 
\begin{equation} \label{cauchy-res-perron-slut}
\sum_{h \leqslant T^{u}/md} 
\frac{B_{-\gamma,\beta,-\alpha,\delta}(hd)}{h} 
= \mathrm{Res}_{s = 0} + \mathrm{Res}_{s = \alpha} 
+ \mathrm{Res}_{s = -\delta} + O(d^{\epsilon}),
\end{equation}
where of course we are referring to the residues of the integrand in (\ref{euler-evaluatiooon-IV}).

Finally we have (almost) got all the puzzle-pieces needed to conclude the proof under the assumption (\ref{extra-condition}). We consider the various factors in (\ref{tredje-termen-i-bevis-II}). First of all 
\begin{align} \label{w-int-ungefar}
\int_{-\infty}^{\infty} w(t) 
\Big( \frac{t}{2 \pi} \Big)^{-\alpha - \gamma} \; dt
&= T^{-\alpha - \gamma} \int_{T/2}^{4T} w(t) 
\exp \Big\{ (-\alpha - \gamma) \log \Big(\frac{t}{2 \pi T} \Big) \Big\} \; dt \nonumber \\ 
& = T^{-\alpha - \gamma} \int_{T/2}^{4T} w(t) 
\Big\{1 + O\Big(\frac{1}{\log{T}}\Big)\Big\} \; dt \nonumber \\
& = T^{-\alpha - \gamma} \int_{-\infty}^{\infty} w(t) \; dt
+ O\Big(1 \cdot \int_{T/2}^{4T} 
\frac{|w(t)|}{\log{T}} \; dt \Big) \nonumber \\
& = T^{-\alpha - \gamma} \int_{-\infty}^{\infty} w(t) \; dt
+ O\Big(\frac{T}{\log{T}} \Big) \nonumber \\
& = T^{-(\alpha + \beta + \gamma + \delta)/2} 
\cdot T^{(\beta - \gamma + \delta -\alpha)/2} 
\cdot \int_{-\infty}^{\infty} w(t) \; dt
+ O\Big(\frac{T}{\log{T}} \Big),
\end{align}
where we for future need also remark that the main term in (\ref{w-int-ungefar}) trivially is $\ll T.$

Secondly, 
\begin{align} \label{A-euler-forenkling}
&A_{-\gamma,\beta,-\alpha,\delta}(0) = \frac{\zeta(1 -\gamma - \alpha) 
\zeta(1 - \gamma + \delta) \zeta(1 + \beta - \alpha) 
\zeta(1 + \beta + \delta) }
{\zeta(2 -\gamma + \beta -\alpha + \delta)} \nonumber \\
&= \frac{1}{\zeta(2)} \cdot \frac{1}{(-\gamma - \alpha)} 
\cdot \frac{1}{(- \gamma + \delta)} \cdot \frac{1}{(\beta - \alpha)} 
\cdot \frac{1}{(\beta + \delta)} + O( \log^{3}{T} ), 
\end{align}
where we remark that the main term in (\ref{A-euler-forenkling}) obviously is of order $\log^{4}{T}.$

Thirdly, the $m$-summands in (\ref{tredje-termen-i-bevis-II}) contain two sums like in (\ref{cauchy-res-perron-slut}). Each of those sums will be evaluated by using (\ref{cauchy-res-perron-slut}) and hence give rise to residues. We will now explain how to treat the residue at $s = \alpha$ in (\ref{cauchy-res-perron-slut}). Explicitly the residue is 
\begin{equation} \label{residue-explicit}
\frac{\zeta(1 + \delta + \alpha) (T^{u}/md + \mu)^{\alpha} 
G(1 + \alpha,-\gamma,\beta,-\alpha,\delta,d)}{\alpha}.
\end{equation}
Repeatedly using the Theorem of Calculus 
\[ F(\gamma(b)) - F(\gamma(a)) = \int\limits_{\gamma} F'(z) \, dz, \]
we obtain, recalling (\ref{g-diff-bound}), that 
\begin{equation} \label{G-liten-skillnad}
G(1 + \alpha,-\gamma,\beta,-\alpha,\delta,d) 
= G(1,0,0,0,0,d) + O\Big( \frac{d^{\epsilon}}{\log{T}} \Big). 
\end{equation}
We easily deduce that (\ref{residue-explicit}) is 
\begin{equation} \label{residue-explicit-II}
\frac{(T^{u}/md)^{\alpha} 
G(1,0,0,0,0,d)}
{(\delta + \alpha) \alpha} + O(d^{\epsilon} \log{T}), 
\end{equation}
where clearly the main term in (\ref{residue-explicit-II}) is 
\[ \ll d^{\epsilon} \log^{2}{T}. \]
The other five residues are treated similarly. Then we put in our new expressions for 
$\sum_{h \leqslant T^{u}/md} \frac{B_{-\gamma,\beta,-\alpha,\delta}(hd)}{h}$ 
and $\sum_{k \leqslant T^{u}/md} \frac{E_{-\gamma,\beta,-\alpha,\delta}(kd)}{k}$ 
into (\ref{tredje-termen-i-bevis-II}). Since the outer sum over $d$ in (\ref{tredje-termen-i-bevis-II}) always will be convergent (due to the presence of $d^2$ in the denominator), an inspection yields two things. First of all that if we ever choose to take the error-part of any of the discussed expressions, we will in (\ref{tredje-termen-i-bevis-II}) end up with something that is 
$\ll T\log^{8}{T}$ and which thus can be relegated to the error term in (\ref{simplified-enkla-fallet-ekv}). And secondly that we may replace the 
$(T^{u}/md)^{\alpha}$ in (\ref{residue-explicit-II}) by $(T^{u}/m)^{\alpha}$ and change the range of summation over $m$ in (\ref{tredje-termen-i-bevis-II}) from $m \leqslant T^{u}/d$ to $m \leqslant T^{u},$ since the change introduced by these is $\ll T\log^{8}{T}.$

All-in-all we get that (\ref{tredje-termen-i-bevis-II}) equals 
\begin{align} \label{tredje-termen-i-bevis-III}
&\Big[ \int_{-\infty}^{\infty} w(t) \; dt \Big] 
\cdot \frac{1}{\zeta(2)} 
\cdot \sum_{d \leqslant T^{u}} \frac{\mu(d) G(1,0,0,0,0,d)^{2}}{d^{2}} \\
&\cdot T^{-(\alpha + \beta + \gamma + \delta)/2} 
\cdot \sum_{m \leqslant T^{u}} \frac{1}{m} 
\mathcal{Q}(\{ \beta, -\gamma \},\{ \delta, -\alpha \},T,T^{u}/m,f,f) 
+ O(T \log^{8}{T}). \nonumber
\end{align}

We may extend the finite sum over $d$ in (\ref{tredje-termen-i-bevis-III}) to an infinite one, since 
\begin{equation} \label{d-finite-infinite}
\sum_{d > T^{u}} \frac{\mu(d) G(1,0,0,0,0,d)^{2}}{d^{2}} 
\ll \sum_{d > T^{u}} \frac{1 \cdot d^{2 \epsilon}}{d^{2}} 
\ll \frac{1}{\log{T}}. 
\end{equation} 
It then remains to establish the identity 
\begin{equation} \label{d-identity}
\frac{1}{\zeta(2)} 
\sum_{d = 1}^{\infty} \frac{\mu(d) G(1,0,0,0,0,d)^{2}}{d^{2}}
= \displaystyle\prod_{p} 
\Big\{ \Big(1 + \frac{4}{p} + \frac{1}{p^{2}}\Big)
\Big(1 - \frac{1}{p}\Big)^{4} \Big\}. 
\end{equation}
By an explicit calculation 
\begin{equation} \label{G-forsta-steg}
G(1,0,0,0,0,1) =\prod_{p} \Big\{ (1-p^{-1})^2 
\displaystyle\sum_{M = 0}^{\infty} 
\frac{B_{0,0,0,0}(p^M)}{p^{M}} \Big\} 
= \prod_{p} \Big\{ \frac{(1-p^{-1})(1+2p^{-1})}{(1+p^{-1})} \Big\}. 
\end{equation} 
It is natural to define 
\begin{equation} \label{g-lilla-def}
g(1,0,0,0,0,d):= G(1,0,0,0,0,d)/G(1,0,0,0,0,1) 
\end{equation}
and next we show that $g(1,0,0,0,0,d)$ is a multiplicative function. 
Let $D_1,$ $D_2 \in \mathbb{N}$ be relatively prime. Recall (\ref{g-split}) and note that $B_{0,0,0,0}(n)$ is a multiplicative function. By expanding out the terms involved in the relation 
\[ g(1,0,0,0,0,D_1) \cdot g(1,0,0,0,0,D_2) = g(1,0,0,0,0,D_1D_2), \]
this equality is, after cancellation, seen to be equivalent to 
\[ \displaystyle\sum_{\substack{h = 1 \\ p|h \Rightarrow p|D_1}}^{\infty} 
\frac{B_{0,0,0,0}(hD_1)}{h} \cdot 
\displaystyle\sum_{\substack{k = 1 \\ p|k \Rightarrow p|D_2}}^{\infty} 
\frac{B_{0,0,0,0}(kD_2)}{k}
= \displaystyle\sum_{\substack{m = 1 \\ p|m \Rightarrow p|D_1D_2}}^{\infty} 
\frac{B_{0,0,0,0}(mD_1D_2)}{m}. \]
This identity can be seen to hold by equalling denominators (using multiplicativity of the function $B_{0,0,0,0}(n)$).

Therefore the LHS of (\ref{d-identity}) equals 
\begin{equation} \label{d-identitets-mellansteg}
\frac{G(1,0,0,0,0,1)^{2}}{\zeta(2)} 
\sum_{d = 1}^{\infty} \frac{\mu(d) g(1,0,0,0,0,d)^2}{d^2}. 
\end{equation} 
Using multiplicativity we are lead to studying  
\begin{equation} \label{d-identitets-steg-tva}
\sum_{M = 0}^{\infty} \frac{\mu(p^M) g(1,0,0,0,0,p^M)^2}{p^{2M}} 
= 1 - \frac{g(1,0,0,0,0,p)^2}{p^{2}}. 
\end{equation} 
An explicit calculation gives that 
\begin{equation} \label{d-identitets-steg-tre}
g(1,0,0,0,0,p) = \frac{2 + p^{-1}}{1 + 2p^{-1}}.
\end{equation}
Upon using this, the result in (\ref{d-identity}) follows, since 
\begin{equation*} \label{zeta-value}
\frac{1}{\zeta(2)} = \prod_{p} (1 - p^{-2}) 
\end{equation*}
and this completes the proof of Theorem \ref{simplified-enkla-fallet} under the assumption of (\ref{extra-condition}).

Let us write $\alpha = \alpha_{1},$ $\beta = \alpha_{2},$ $\gamma = \alpha_{3}$ and $\delta = \alpha_{4}.$ Suppose that Theorem \ref{simplified-enkla-fallet} is to be proved for all $|\alpha_{i}| \leqslant C/\log{T}.$ Consider now (without the extra assumption (\ref{extra-condition})) any $|\alpha_{i}| \leqslant C/\log{T}.$ 
The idea is to use Cauchy's integral formula in order to go from the previous ``easier" case to the general case. 

Both the LHS and the main term in the RHS of Theorem \ref{simplified-enkla-fallet} are analytic functions of the complex variables $\alpha_1$, $\alpha_2$, $\alpha_3$, $\alpha_4$ (recall Remarks \ref{analyticity-remark} and \ref{analyticity-HL-remark}), let us denote them by 
L$(\alpha_1, \alpha_2, \alpha_3, \alpha_4)$ and 
R$(\alpha_1, \alpha_2, \alpha_3, \alpha_4)$ respectively. 
Let $D$ be the polydisc defined as the Cartesian product of the open discs $D_{i},$ i.e.\ $D = \prod_{i = 1}^{4} D_{i},$ where 
\[D_{i} := \{s \in \mathbb{C} : |s - \alpha_{i}| < r_{i}\}, \]
with  
\begin{equation*}  
r_{i} = \frac{2^{i+1}C}{\log{T}}. 
\end{equation*}
An application of Cauchy's integral formula yields that
\begin{align} \label{cauchy-steg-sixth-artttt}
&\textrm{L}(\alpha_1, \alpha_2, \alpha_3, \alpha_4) 
- \textrm{R}(\alpha_1, \alpha_2, \alpha_3, \alpha_4) \nonumber \\
&=\frac{1}{(2\pi i)^4} \int \cdots \int 
\int_{\partial D_{1} \times \cdots \times \partial D_{4}}
\frac{\textrm{L}(\beta_1, \beta_2, \beta_3, \beta_4) 
- \textrm{R}(\beta_1, \beta_2, \beta_3, \beta_4)}
{(\beta_{1} - \alpha_{1}) \cdots (\beta_{4} - \alpha_{4})} 
\, d\beta_{1} ... d\beta_{4}.
\end{align}
Now we notice that the $\beta_i$ satisfy $|\beta_i - \alpha_i| = r_i,$ 
which is easily seen to imply that $\beta_i \ll 1/\log{T}$ and that 
\[ |\beta_i| \geqslant 2C/\log{T}, \; \; \; |\beta_i + \beta_j| \geqslant 2C/\log{T} \; \; \; \; \textrm{and} \; \; \; |\beta_i - \beta_j| \geqslant 2C/\log{T}. \] 
This theorem thus applies if the $\beta_i$-terms are seen as shifts, so that we have 
\begin{equation} \label{repeat-diff-result}
\textrm{L}(\beta_1, \beta_2, \beta_3, \beta_4) 
- \textrm{R}(\beta_1, \beta_2, \beta_3, \beta_4) 
\ll T\log^{8}{T}.
\end{equation}
By using (\ref{repeat-diff-result}) and considering the trivial upper bound for (\ref{cauchy-steg-sixth-artttt}), the latter is $\ll T\log^{8}{T}.$ This finally completes the proof of Theorem \ref{simplified-enkla-fallet}.
\end{proof}

\subsection{Further simplification of Theorem \ref{sats1} in the most standard case} 
\label{subsection-further-simplification}

Let us finish the discussion of how to evaluate the integral in (\ref{integral-namn-A}). We take\footnote{We will later let $\lambda \to 0.$}  
\begin{equation} \label{shifts-val}
\{ \alpha, \beta, \gamma, \delta \} 
= \Big(\frac{i}{\log{T}}\Big) 
\{ \kappa + \lambda, -\lambda, -\lambda, -\kappa + \lambda \}
\end{equation}
and apply Theorem \ref{simplified-enkla-fallet}. One obtains an expression involving sums over $m$ (see (\ref{simplified-enkla-fallet-ekv})). We now show that such sums can be dealt with by using\footnote{To make sense of the RHS in the formula in the case $\alpha = 0,$ use the $\alpha^{0}$-coefficient in the Taylor series.} 
\begin{equation} \label{m-sum-ekv-noll}
\sum_{m \leqslant T^{u}} \frac{(T^{u}/m)^{\alpha}}{m} \approx 
\frac{(T^{\alpha u} - 1)}{\alpha}, 
\end{equation}
that is to say that after having done these replacements in all of the sums in the RHS of Theorem \ref{simplified-enkla-fallet}, the arising version of (\ref{simplified-enkla-fallet-ekv}) is true. 

We first prove this claim under the extra condition (\ref{extra-condition}). By partial summation we have for $t \ll T$ that 
\begin{equation} \label{m-sum-ekv-apostol}
\sum_{m \leqslant t} \frac{1}{m^{1 + \alpha}} 
= [t] \cdot t^{-(1 + \alpha)} 
+ (1 + \alpha) \int_{1}^{t} [w] \cdot w^{-2 - \alpha} \; dw 
= \frac{(1 - t^{-\alpha})}{\alpha} + O(1). 
\end{equation}
It follows immediately that 
\begin{equation} \label{m-sum-ekv-noll-exakt}
\sum_{m \leqslant T^{u}} \frac{(T^{u}/m)^{\alpha}}{m} 
= T^{\alpha u} \sum_{m \leqslant T^{u}} \frac{1}{m^{1+\alpha}} 
= \frac{(T^{\alpha u} - 1)}{\alpha} + O(1). 
\end{equation}
Using (\ref{m-sum-ekv-noll-exakt}) and trivial estimates, we reach our conclusion (i.e.\ the error parts in (\ref{m-sum-ekv-noll-exakt}) can be absorbed into the error term in (\ref{simplified-enkla-fallet-ekv})).

In order to retrieve the general case from this special case, we work exactly as in the proof of Theorem \ref{simplified-enkla-fallet}. However, in order to apply Cauchy's integral formula we must first show analyticity of the new RHS of (\ref{simplified-enkla-fallet-ekv}). To this end, we apply Lemma 2.5.1 in \cite{CFKRS}. Ensuring that the conditions for applying the latter are met here essentially boils down to checking that what one gets after using\footnote{Read this as doing the relevant replacements.} (\ref{m-sum-ekv-noll}) on 
\begin{align} \label{appendix-replacement}
&\sum_{m \leqslant T^{u}} 
\frac{f(x_{1},x_{2},T^{u}/m) f(x_{3},x_{4},T^{u}/m)}{m} \\
& = \sum_{m \leqslant T^{u}} \bigg\{ \frac{1}{m} 
\Big( \frac{1}{x_{1}x_{2}} - \frac{(T^{u}/m)^{-x_{1}}}
{x_{1}(x_{2} - x_{1})} - \frac{(T^{u}/m)^{-x_{2}}}
{x_{2}(x_{1} - x_{2})} \Big)
\Big( \frac{1}{x_{3}x_{4}} - \frac{(T^{u}/m)^{-x_{3}}}
{x_{3}(x_{4} - x_{3})} - \frac{(T^{u}/m)^{-x_{4}}}
{x_{4}(x_{3} - x_{4})} \Big) \bigg\} \nonumber
\end{align}
is an analytic function in terms of shifts 
$x_{1},$ $x_{2},$ $x_{3}$ and $x_{4}.$ 
Explicitly one ends up with 
\begin{align} \label{appendix-replacementsvar} 
&\frac{\log(T^{u})}{x_{1} x_{2} x_{3} x_{4}} 
+ \frac{(T^{-u x_{3}} - 1)}{x_{1} x_{2} x_{3}^{2} (x_{4} - x_{3})}
+ \frac{(T^{-u x_{4}} - 1)}{x_{1} x_{2} x_{4}^{2} (x_{3} - x_{4})}
+ \frac{(T^{-u x_{1}} - 1)}{x_{1}^{2} (x_{2} - x_{1}) x_{3} x_{4}}
+ \frac{(T^{-u x_{2}} - 1)}{x_{2}^{2} (x_{1} - x_{2}) x_{3} x_{4}} \nonumber \\ 
&- \frac{(T^{-u(x_{1} + x_{3})} - 1)}
{x_{1} x_{3} (x_{2} - x_{1}) (x_{4} - x_{3}) (x_{1} + x_{3})} 
- \frac{(T^{-u(x_{1} + x_{4})} - 1)}
{x_{1} x_{4} (x_{2} - x_{1}) (x_{3} - x_{4}) (x_{1} + x_{4})} \nonumber \\ 
&- \frac{(T^{-u(x_{2} + x_{3})} - 1)}
{x_{2} x_{3} (x_{1} - x_{2}) (x_{4} - x_{3}) (x_{2} + x_{3})} 
- \frac{(T^{-u(x_{2} + x_{4})} - 1)}
{x_{2} x_{4} (x_{1} - x_{2}) (x_{3} - x_{4}) (x_{2} + x_{4})}, 
\end{align}
which admittedly does not look too pleasant at first sight. However, there is a clever and natural strategy to employ to realise why (\ref{appendix-replacementsvar}) has to be analytic in 
$x_{1},$ $x_{2},$ $x_{3}$ and $x_{4}.$ 

Consider\footnote{To arrive at (\ref{appendix-analytic}), essentially one first uses Remark \ref{anal-of-f-standard} on both the expressions in round brackets in (\ref{appendix-replacement}), then moves the sum over $m$ inside the double-integral and finally applies (\ref{m-sum-ekv-noll}).} 
\begin{equation} \label{appendix-analytic}
\Psi(x_{1},x_{2},x_{3},x_{4}) := 
\frac{1}{(2 \pi i)^{2}} \int_{R_{2}} \int_{R_{1}} 
\frac{(T^{u(s + w)} - 1)}{(s + x_{1})(s + x_{2})s(w + x_{3})(w + x_{4})w(s + w)} 
\; ds \; dw, 
\end{equation}
where $R_{1}$ and $R_{2}$ denote counter-clockwise rectangular paths, with vertices at 
$\pm 2 \pm 2i$ and $\pm 1 \pm i$ respectively. Whenever 
\begin{equation} \label{appendix-CA-extra-condition}
 x_i \neq 0 \; \; \; \textrm{and } \; x_i \neq \pm x_j 
\end{equation}
is satisfied, one notices that $\Psi(x_{1},x_{2},x_{3},x_{4})$ equals (\ref{appendix-replacementsvar}), by carefully applying Cauchy's Residue Theorem twice. Also, trivial estimates give that $\Psi(x_{1},x_{2},x_{3},x_{4})$ is bounded. Therefore we may conclude that all the possible singularities of (\ref{appendix-replacementsvar}) are removable. 

In order to evaluate (\ref{integral-namn-A}) we thus use Theorem \ref{simplified-enkla-fallet} and proceed as explained above by using (\ref{m-sum-ekv-noll}). We then substitute in (\ref{shifts-val}) and view our answer as a Laurent series in terms of $\lambda.$ Since the LHS of (\ref{simplified-enkla-fallet-ekv}) remains bounded as $\lambda \to 0,$ we must have cancellation so that our Laurent series actually is a Taylor series. Since we are letting $\lambda \to 0$ anyway, what all this means is that in practice one focuses term-wise on finding just the $\lambda^{0}$-coefficients\footnote{Here the use of Mathematica was helpful.}. 

Doing this gives us an answer in terms of $\kappa.$ By seeing various symmetries in the calculations one can both save time and simplify the answer\footnote{On a related note, by Remark \ref{cancellation-forstas} one could focus on finding just the analytic part of each term. Although this would save time, one advantage of keeping track of the negative $\kappa$-powers is that if all their coefficients cancel in the (total) answer, then it is ``likely" that one has not made any calculation-errors!}. For example in the present integral-calculation one can spot that the contributions from the terms originating from the first and second term in (\ref{Hy-integrand-uttrycket}) are complex conjugates\footnote{Similarly one here pairs together the terms corresponding to the third and sixth term, and the fourth and fifth term in (\ref{Hy-integrand-uttrycket}).}. This will mean that via use of Euler's formula 
\begin{equation} \label{euler-trig}
\exp(ix) = \cos{x} + i \sin{x}, 
\end{equation}
one obtains nice trigonometric terms in the answer (see (\ref{integral-namn-A-svarrrr})). 

\subsection{Simplified versions of Theorem \ref{sats1} in two other cases} 
\label{subsection-simplified-version}

Let us recall the notation 
\begin{equation*} 
M(s) = \sum_{h \leqslant T^{u}} \frac{1}{h^{s}} 
\end{equation*}
and 
\begin{equation*} 
N(s) = \sum_{h \leqslant T^{u}} \frac{\log(T^{u}/h)}{h^{s}}.
\end{equation*}
Let us also recall the notation 
\begin{equation*} 
f(x_{1},x_{2},T_{2}) 
= \frac{1}{x_{1} x_{2}} - \frac{T_{2}^{-x_{1}}}{x_{1}(x_{2} - x_{1})} 
- \frac{T_{2}^{-x_{2}}}{x_{2}(x_{1} - x_{2})} 
\end{equation*}
and define 
\begin{equation} \label{g-function-spec-ett}
g(x_{1},x_{2},T_{2}) 
:= \frac{\log{T_{2}}}{x_{1} x_{2}} 
- \frac{1}{x_{1}^{2} x_{2}}
- \frac{1}{x_{2}^{2} x_{1}}
+ \frac{T_{2}^{-x_{1}}}{x_{1}^{2}(x_{2} - x_{1})} 
+ \frac{T_{2}^{-x_{2}}}{x_{2}^{2}(x_{1} - x_{2})}. 
\end{equation}
The following two theorems can be proved very similarly to Theorem \ref{simplified-enkla-fallet}.  

\begin{theorem} \label{simplified-m-n-sats}
With the same assumptions as in Theorem \ref{simplified-enkla-fallet}, we have 
\begin{align} \label{simplified-m-n-ekv}
&\int_{-\infty}^{\infty} M(\tfrac{1}{2} + it) N(\tfrac{1}{2} - it) 
\zeta(\tfrac{1}{2} + \alpha + it) \zeta(\tfrac{1}{2} + \beta + it) 
\zeta(\tfrac{1}{2} + \gamma - it) \zeta(\tfrac{1}{2} + \delta - it) 
w(t) \; dt \nonumber \\
&= \Big[ \int_{-\infty}^{\infty} w(t) \; dt \Big] 
\cdot a_{3} 
\cdot T^{-(\alpha + \beta + \gamma + \delta)/2} 
\cdot \sum_{m \leqslant T^{u}} \frac{1}{m} 
\mathcal{Q}_{A,B}(T,T^{u}/m,g,f) + O(T \log^{9}{T}), 
\end{align} 
where we still use (\ref{A-val-set}) and (\ref{B-val-set}).
\end{theorem}

\begin{theorem} \label{simplified-n-n-sats}
With the same assumptions as in Theorem \ref{simplified-enkla-fallet}, we have 
\begin{align} \label{simplified-n-n-ekv}
&\int_{-\infty}^{\infty} |N(\tfrac{1}{2} + it)|^{2} 
\zeta(\tfrac{1}{2} + \alpha + it) \zeta(\tfrac{1}{2} + \beta + it) 
\zeta(\tfrac{1}{2} + \gamma - it) \zeta(\tfrac{1}{2} + \delta - it) 
w(t) \; dt \nonumber \\
&= \Big[ \int_{-\infty}^{\infty} w(t) \; dt \Big] 
\cdot a_{3} 
\cdot T^{-(\alpha + \beta + \gamma + \delta)/2} 
\cdot \sum_{m \leqslant T^{u}} \frac{1}{m} 
\mathcal{Q}_{A,B}(T,T^{u}/m,g,g) + O(T \log^{10}{T}), 
\end{align} 
where we again still use (\ref{A-val-set}) and (\ref{B-val-set}).
\end{theorem}

Similarly to how (\ref{m-sum-ekv-noll}) was used, one may handle the above sums over $m$ by using 
\begin{equation} \label{m-sum-ekv-ett}
\sum_{m \leqslant T^{u}} \frac{(T^{u}/m)^{\alpha} \log(T^{u}/m)}{m} \approx 
\frac{1}{\alpha^{2}} - \frac{T^{\alpha u}}{\alpha^{2}} 
+ \frac{T^{\alpha u} \log(T^{u})}{\alpha}
\end{equation}
and 
\begin{equation} \label{m-sum-ekv-tva}
\sum_{m \leqslant T^{u}} \frac{(T^{u}/m)^{\alpha} \log^{2}(T^{u}/m)}{m} \approx 
- \frac{2}{\alpha^{3}} + \frac{2 T^{\alpha u}}{\alpha^{3}} 
- \frac{2 T^{\alpha u} \log(T^{u})}{\alpha^{2}} 
+ \frac{T^{\alpha u} \log^{2}(T^{u})}{\alpha},
\end{equation}
these two formulas arising by applying partial summation to (\ref{m-sum-ekv-apostol}). 

\subsection{Differentiation} 
\label{subsection-differentiation}

Some of our integrals in (\ref{integral-namn-A})-(\ref{integral-namn-J}) involve differentiation. This is no problem though, as it is possible to differentiate our Theorems \ref{simplified-enkla-fallet}, \ref{simplified-m-n-sats} and \ref{simplified-n-n-sats} with respect to any of the shifts. To do this, we simply use ``Cauchy's integral trick" and the result follows immediately. An illustration of this is now done, namely in the case when we differentiate Theorem \ref{simplified-enkla-fallet} once with respect to $\alpha.$ The conclusion in this case is as follows: 

\begin{theorem} \label{simplified-enkla-fallet-diffat}
With notation and assumptions as in Theorem \ref{simplified-enkla-fallet}, \begin{align} \label{simplified-enkla-fallet-diffat-ekv}
&\int_{-\infty}^{\infty} |M(\tfrac{1}{2} + it)|^{2} 
\zeta'(\tfrac{1}{2} + \alpha + it) \zeta(\tfrac{1}{2} + \beta + it) 
\zeta(\tfrac{1}{2} + \gamma - it) \zeta(\tfrac{1}{2} + \delta - it) 
w(t) \; dt \nonumber \\
&= \Big[ \int_{-\infty}^{\infty} w(t) \; dt \Big] 
\cdot a_{3} 
\cdot \frac{\partial}{\partial\alpha} 
\Big[ T^{-(\alpha + \beta + \gamma + \delta)/2} 
\sum_{m \leqslant T^{u}} \frac{1}{m} 
\mathcal{Q}_{A,B}(T,T^{u}/m,f,f) \Big] 
+ O(T \log^{9}{T}). 
\end{align}
\end{theorem}

\begin{proof}
Beginning with Theorem \ref{simplified-enkla-fallet}, we know that the LHS and the main term in the RHS of (\ref{simplified-enkla-fallet-ekv}) are analytic functions of the complex variables $\alpha$, $\beta$, $\gamma,$ $\delta.$ Take the first one and subtract the latter and we get an analytic function, let us call it 
$D(\alpha, \beta, \gamma, \delta).$ Theorem \ref{simplified-enkla-fallet} tells us that 
\[ D(\alpha, \beta, \gamma, \delta) \ll T \log^{8}{T}. \] 
Now use Cauchy's integral formula for the derivative with a radius 
$r = 1/\log{T},$ i.e.\ 
\[ \frac{\partial D(\alpha, \beta, \gamma, \delta)}{\partial \alpha} 
= \frac{1}{2\pi i} \int_{|w - \alpha| = r}
\frac{D(w, \beta, \gamma, \delta)}{(w - \alpha)^2} \; dw. \]
Then the pathlength is of order $(\log{T})^{-1}$ and the integrand is trivially 
\[ \ll (\log{T})^2 \cdot T \log^{8}{T} 
= T \log^{10}{T}, \] 
giving 
\[ \frac{\partial D}{\partial\alpha} \ll T \log^{9}{T}. \]
To finish the proof, remember that clearly the derivative of the difference of two analytic functions equals the difference of the derivatives of those two functions. 
\end{proof}

\section*{Acknowledgement}

I would like to thank my supervisor Roger Heath-Brown for his excellent guidance and help.

\end{document}